\documentclass{amsart}
\usepackage{color}
\usepackage{graphicx}
\usepackage{hyperref}
\usepackage{amssymb}
\usepackage{amsfonts}
\usepackage{euscript}
\usepackage[all]{xy}
\usepackage{mathabx}
\usepackage{mathtools}
\usepackage{bbm}
\usepackage{enumitem}

\numberwithin{equation}{section}
\numberwithin{figure}{section}
\renewcommand{\subsection}{\hspace{-\parindent}\refstepcounter{subsection}{\bf \arabic{section}\alph{subsection}. }}

\theoremstyle{plain}
\newtheorem{thm}{Theorem}[section]
\newtheorem{theorem}[thm]{Theorem}

\newtheorem{corollary}[thm]{Corollary}

\newtheorem{assumption}[thm]{Assumption}
\newtheorem{definition}[thm]{Definition}

\newtheorem{remark}[thm]{Remark}

\newtheorem{non-example}[thm]{Non-example}

\newtheorem{properties}[thm]{Properties}

\newtheorem{lemma}[thm]{Lemma}

\newtheorem*{claim*}{Claim} 
\newtheorem*{lemma*}{Lemma}

\newtheorem*{theorem*}{Theorem}
\newtheorem*{conjecture*}{Conjecture}

\newcommand{\bC}{{\mathbb C}}

\newcommand{\bR}{{\mathbb R}}

\newcommand{\bZ}{{\mathbb Z}}

\title[$S^1$-localisation by pseudocycles, and symplectic cohomology]{$S^1$-localisation by pseudocycles, lifts to $S^1$-localisation of moduli spaces, and application to invariants of $S^1$-equivariant symplectic cohomology}
\author{Nicholas Wilkins}
\date{\today}

\begin{document}

\maketitle

\begin{abstract}
    We demonstrate a way to apply $S^1$-localisation to moduli spaces of holomorphic curves. We first prove a reinterpretation of Atyiah-Bott $S^1$-localisation, called {\it localisation by pseudocycles} (LbP), for a smooth semifree $S^1$-action on a manifold. We demonstrate that, for certain moduli spaces of holomorphic curves parametrised by some stratum of the homotopy quotient of a manifold, we may ``lift" the LbP procedure from the parameter space to the moduli space. As an application we deduce relations between equivariant symplectic classes and Gromov-Witten invariants, thus proving a conjecture of Seidel.
\end{abstract}

\section*{Acknowledgements}
I would like to thank Paul Seidel, Shaoyun Bai, Dan Pomerleano, Mark McLean, Aleksey Zinger for helpful conversations. I would like to thank Alex Pieloch for comments on the previous version of this paper.

This research was funded by the Simons Foundation, through award 652299.

\section{Introduction}
\subsection{Background}

When studying a topological space $A$ with an $S^1$-action, the first way one may try to understand its structure is by calculating the $S^1$-equivariant cohomology, i.e. $H^*_{S^1}(A) := H^*(ES^1 \times_{S^1} A)$. We call $ES^1 \times_{S^1} A$ the {\it homotopy quotient of $A$}, and $ES^1$ in this instance is taken as the embedded $S^{\infty} \subset \bC^{\infty}$. This cohomology is generally quite difficult to compute, however, even if $H^*(A)$ is known.

In the case where $A$ is a smooth manifold, there is a strong technique due to Atiyah-Bott \cite{atiyahbott} (recalled in Theorem \ref{theorem:s1-localisation}), often named Atiyah-Bott $S^1$-localisation, which broadly states the following: ``$H^*_{S^1}(A)$ is mostly recovered from $H^*(F)$", where $F \subset A$ is the fixed-point set. We use the term ``mostly" because one also needs to understand the equivariant embedding of $F$ in $A$, but the arising correction is somehow ``small" in the equivariant cohomology. In particular, it is torsion. 

While this is a great technique to understand the equivariant topology of a manifold $A$, if one is working with (a $\bZ_{\ge 0}-$graded collection of) moduli spaces that one counts to define invariants of $S^1$-equivariant symplectic cohomology (throughout, we will colloquially refer to these as: {\it $S^1$-equivariant moduli spaces}), then because of the need to ensure regularity, one very rarely encounters actual homotopy quotients. In particular, the moduli spaces one constructs do fibre over substrata of $BS^1$, but are not themselves substrata of a homotopy quotient. Nonetheless, we would often like to apply a form of $S^1$-localisation to such moduli spaces.

What we present in this paper is a reinterpretation of Atiyah-Bott $S^1$-localisation, described in Section \ref{sec:LbP}. This reinterpretation has the following useful property: if some $S^1$-equivariant moduli space is parametrised by the nested substrata of a homotopy quotient, along with certain conditions to ensure ``niceness", then one can apply the localisation by pseudocycles (throughout denoted LbP) procedure to the homotopy quotient substrata, and lift this to the moduli space. This procedure is detailed in Section \ref{sec:lifting-to-moduli-spaces}. We then use this to prove a conjecture due to Seidel, i.e. \cite[Conjecture 3.3]{lefschetz3}, in Section \ref{sec:ourexample}. This theorem yields a relationship between Gromov-Witten invariants and equivariant Borman-Sheridan classes. The nomenclature due to unpublished work of the named authors, but these and similar classes were studied in \cite{ganatra-pomerleano}, \cite{borman-sheridan-varolgunes}, \cite{lefschetz3}, \cite{tonkonog}, and \cite{ganatra2021log}.

\subsection{Main results}

Throughout, $A$ will denote a smooth compact manifold equipped with a semi-free $S^1$-action. Its fixed-point set is denoted $F$. Associated to any closed $S^1$-invariant submanifold $N \subset A$, and any $i \ge 0$, there is a pseudocycle $\mu^N_i$ of $S^{2i+1}\times_{S^1} A$ denoting an element  $[\mu^N_i] \in H_{2i + \dim (N)}^{S^1}(A)$ as in Equation \eqref{equation:proj-inc}.

In Section \ref{sec:LbP}, we demonstrate in detail the following relationships between such cycles, i.e. Equation \eqref{equation:lbp1} (or \eqref{equation:lbp2} if $A$ has a boundary):

\begin{theorem*}[\ref{theorem:lbp}]
If $\partial A = \emptyset$ then:
$$ \sum_{k \ : \ \lambda_k = 2} \epsilon(F_k) [\mu_i^F] + \sum_k n_k [\mu_i^{BD(p^{-1}B_k)^o}] - [\mu_{i-1}^A] = 0,$$ or if $\partial A \neq \emptyset$ and Assumption \ref{assumption:mfdwbdry-assumption} holds:  $$ \sum_{k \ : \ \lambda_k = 2}  \epsilon(F_k) [\mu_i^{F_k}] + \sum_k n_k [\mu_i^{BD(p^{-1}B_k)^{wc}} \# \mu_{i+1}^{\partial B_k}] - [\mu^A_{i-1} \#_{\mu^N_{i-1}} \mu^{N/S^1}_{i}] = 0,$$ where the $\epsilon(F_k) \in \{ -1,1\}$ and $n_k \in \bZ$. 

\end{theorem*}

Here, the $F_k$ are connected components of the $S^1$-fixed-point set of $A$, the nonnegative integers $\lambda_k$ are the codimensions of the respective $F_k$, and the $B_k$ are connected components of the codimension $2$ submanifold $B \subset (A \setminus F) /S^1$, such that the removal of $B$ trivialises the $S^1$-bundle $$p: A \setminus F \rightarrow (A \setminus F) / S^1.$$ Further, $BD$ is the blow-down of the oriented blow-up; alternatively, it is the collapsing of the fibres of a tubular neighbourhood of $F$, thus providing a smooth map $A \setminus F \rightarrow A$.

We next recall that when $E$ arises from a Lefschetz fibration with a fibre removed, with the conditions recalled in Section \ref{subsec:sympl-coh}, we can define $S^1$-equivariant symplectic cochains and $S^1$-equivariant symplectic cohomology, respectively $SC^*_{eq}(E)$ and $SH^*_{eq}(E)$, as in Section \ref{subsec:s1-eq-sympl-coh}. 

In Section \ref{subsec:lifting-LbP}, we use LbP to prove Theorem \ref{theorem:lift-to-moduli}, which states the following: suppose we are given some $S^1$-equivariant chains $\Psi_i \in SC^*_{eq}(E)$, each constructed by counting the number of points in the moduli spaces $\mathcal{M}_{\mu_i^A, i}$, where for each $i$, the moduli space $\mathcal{M}_{\mu_i^A, i}$ consists of perturbed, pseudoholomorphic maps parametrised by $S^{2i+1} \times_{S^1} A$  (wherein $A$ is a space of marked domains equipped with an $S^1$-action). Suppose further that $\Psi'_i \in SC^*_{eq}(E)$ is constructed by counting the number of points in a moduli space $\mathcal{M}_{\mu_i^C, i}$, where $\mathcal{M}_{\mu_i^C, i}$ is a moduli space of holomorphic curves parametrised by $S^{2i+1} \times_{S^1} C$ for $C \subset A$ being the oriented union of the codimension $2$ submanifolds $F_k$, and the $p^{-1}B_k$, as above and in \eqref{equation:lbp1}. We also need to provide some control over our choices to order to ensure that the lifting procedure is well behaved (i.e. we do not get any additional terms than those in Theorem \ref{theorem:lift-to-moduli}). In this case we get equality between the following two cycles in $S^1$-equivariant symplectic cohomology:

\begin{theorem*}[\ref{theorem:lift-to-moduli}]

$$u \sum_i \Psi_i u^{i} = \sum_i \Psi'_i u^i \in SH^*_{S^1}(E).$$
\end{theorem*}

Here $u \in H^2(BS^1;\bZ)$ is the equivariant variable.

Finally, in Section \ref{sec:ourexample} we use repeated applications of Theorem \ref{theorem:lift-to-moduli} in the case of the Borman-Sheridan class to prove the below equality and thus, as a special case (for $n=1,2$), answer a conjecture due to Seidel \cite[Conjecture 3.3]{lefschetz3}.

\begin{theorem*}[\ref{theorem:general-solution}]
\begin{equation*}
\begin{array}{l} n! \cdot \Biggl( u^{n} \tilde{s}^{(n)}_{eq} + PSS_{eq} \biggl(  \sum_{j=2}^{n-1}  u^{2n-j-1} \xi^j \biggr) \Biggr) = \\ PSS_{eq} \Biggl( \sum_{j=1}^{n-1} ((n-j-1)!)^2 \cdot u^{2(n-j)-1} \biggl( \Bigl(\clubsuit^{\circ j-1} (z^{(1)})\Bigr) \cup \Bigl( (n-j) u + [M] \Bigr) \biggr) *_F^{(n-j)} [M]   \Biggr|_E   \Biggr), \end{array}
\end{equation*}
where $\xi^j$ are pseudocycles of Gromov-Witten invariants in $F$ with $j$-fold tangencies, and $\clubsuit$ is the operation on quantum cohomology of $F$ ``cup product with a fibre, then take the degree $1$ quantum product with a fibre".
\end{theorem*}

\begin{remark}
    There are multiple ways to attempt to approach proving Theorem \ref{theorem:general-solution}. One must be careful, however, and ensure for example that the moduli spaces in question are completely parametrised by their underlying space of domains. An example where this is not the case is where one introduces intersection points with tangency conditions, as one might for $n \ge 2$ in the above theorem. In this instance, there is a na\"ive extension of the $n=1$ case that fails, because the moduli spaces have extra geometry arising from the ramification points that is not ``seen" by the underlying parameter space. There are ways around this: one could make the lifting procedure more sophisticated to cope, but in our case we simply apply more iterations of lifting LbP, circumventing this issue without requiring more complicated theory.
\end{remark}

\section{Preliminaries}

\subsection{Background on pseudocycles}
\label{subsec:pseudocycles}
We will recall the definition of pseudocycles reasonably briefly, giving citation to \cite{zinger} for full details. The important point, due to Zinger, is that for any smooth manifold, the free Abelian group generated by pseudocycles, up to pseudocycle bordism (both of which we will define below) matches the integral homology of that manifold. 

We will begin with a smooth oriented manifold $M$. 

\begin{definition}
\label{definition:pseudocycle}
A smooth map $f: X \rightarrow M$ is called a \emph{pseudocycle of dimension $n$} if $X$ is an oriented manifold of dimension $n$, and the omega-limit-set $$\Omega_f := \bigcap_{\begin{array}{c} K \subset X \\ K \text{ compact} \end{array}} \overline{f(X \setminus K)}$$ is covered by a smooth map $g: X' \rightarrow M$ where $\dim X' \le n-2$, and the image of $f$ is precompact.  
\end{definition}

The omega-limit-set can be thought of as the ``boundary" of $f$. In general, we say that the omega-limit-set $\Omega_f$ has {\it dimension at most $m$} if it is covered by smooth maps from manifolds of dimension at most $m$.

\begin{definition}
\label{defn:pseudocyclebordism}
Given two pseudocycles $f_0: X_0 \rightarrow M, \ f_1: X_1 \rightarrow M$ of dimension $n$, a smooth map $H: Y \rightarrow N$ is called a \emph {pseudocycle bordism between $f_0$ and $f_1$} if $Y$ is a smooth manifold with boundary, $Y$ is of dimension $n+1$, and the boundary $\partial Y = X_1 \sqcup - X_0$ (the negative sign denoting negative orientation) such that $H|_{X_i} = f_i$. Further, $\Omega_H$ has dimension at most $n-1$.
\end{definition}

The important points that we need from \cite{zinger} are the following:

\begin{itemize}
    \item any closed element $A \in C_n(M;\bZ)$ defines some pseudocycle $\Phi(A): X \rightarrow M$ of dimension $n$, Vice versa any pseudocycle of dimension $n$ defines some closed element of $C_n(M;\bZ)$,
    \item given two closed elements $A,B \in C_n(M;\bZ)$ such that $A-B = dC$, there is a pseudocycle bordism between $\Phi(A)$ and $\Phi(B)$. Vice versa, given two bordant pseudocycles of dimension $n$, the corresponding closed elements of $C_n(M;\bZ)$ represent the same homology class.
\end{itemize}

We also note, without further discussion, that if one removes in Definition \ref{definition:pseudocycle} the condition ``precompact", then the homology group one obtains is locally-finite homology (see \cite{lf-pseudocycles}).

\subsection{Assumptions on $S^1$-actions, and $S^1$-equivariant homology}

In this paper, we will present a method of $S^1$-localisation using pseudocycles. We will need to make some reasonably restrictive assumptions along the way. 

\begin{assumption}[``Nice action" assumption]
\label{assumption:ass1}
Suppose that $A$ is a smooth oriented manifold, equipped with a smooth $S^1$-action, such that for all points of $A$ either $S^1$ acts either freely, or $S^1$ fixes that point. \end{assumption}

In general, this assumption can probably be weakened, but this is irrelevant to the current paper.


We will first discuss equivariant cohomology. Recall that one can model $E S^1$ as $S^{\infty} \subset \mathbb{C}^{\infty}$, and that this can be exhausted by closed submanifolds $S^{\infty} = \bigcup_i S^{2i+1}$. Explicitly, $$S^{2i+1} = \left\{ (z_0,z_1,\dots,z_i,0,0,\dots) \in \mathbb{C}^{\infty} \biggr| \sum_k |z_k|^2 = 1 \right\}.$$ This $ES^1$ is equipped with the diagonal rotation action, $$\theta \in S^1, \ v = (z_0,z_1,z_2,\dots) \in S^{\infty}, \quad \theta \cdot v = (e^{2 \pi i\theta} z_0, e^{2 \pi i\theta} z_1,e^{2 \pi i\theta} z_2,\dots).$$ Then we define $H^*_{S^1}(A) := H^*(S^{\infty} \times_{S^1} A)$. We define the classifying space $BS^1 := ES^1 / S^1$. We notice that there is a map $$\bZ [u ] \cong H^*(BS^1) \rightarrow H^*_{S^1}(A),$$ induced by projection, and so there is a module action of $\bZ [u ]$ on $H^*_{S^1}(A)$. 

Similarly, we define equivariant homology, $H_*^{S^1}(A) := H_*(S^{\infty} \times_{S^1} A)$. Suppose that we fix $i \ge 0$, and some pseudocycle $f: X \xrightarrow{} A$, such that $S^1$ acts on $X$ subject to Assumption \ref{assumption:ass1}, and $f$ is $S^1$-equivariant. Then we will define a natural pseudocycle $\mu^f_i$ in $S^{2i+1} \times_{S^1} A$. In particular, any closed embedded oriented $S^1$-invariant submanifold $C \subset A$ will define a sequence of pseudocycles denoted $\{ \mu^C_i \}_{i =0,1,2,\dots}$, which further represent homology classes (recalling Section \ref{subsec:pseudocycles}) $\{ [C]_i \}_{i=0,1,2,\dots}$.

Each of these odd spheres $S^{2i+1}$ has a nice cell decomposition (albeit not $S^1$-invariant), where one can define $D_{2i} \subset S^{2i+1}$ to be a $2i$-dimensional disc: $$D_{2i} := \left\{ (z_0,z_1,\dots,z_i,0,\dots) \in \mathbb{C}^{\infty} \biggr| \sum_k |z_k|^2 = 1, \ z_i \in \mathbb{R}_{> 0} \right\}.$$ Observe that $\partial \overline{D_{2i}}$ is a copy of $S^{2i-1}$. Further, observe that \begin{equation} \label{equation:d-to-s} S^{2i+1} \setminus \partial \overline{D_{2i}} \cong D_{2i} \times S^1, \end{equation} via an $S^1$-equivariant diffeomorphism. The following diagram commutes:

$$\xymatrix{
D_{2i}
\ar@{->}^-{\cong}[r]
\ar@{->}_-{}[drr]
&
D_{2i} \times \{ 0 \}
\ar@{^{(}->}[r]
&
D_{2i} \times S^1
\ar@{->}^-{\cong}[d]
&
\\
&
&
S^{2i+1} \setminus \partial \overline{D_{2i}}
\ar@{^{(}->}[r]
&
S^{2i+1}
\ar@{->}^-{\pi^{S^{2i+1}}}[d]
\\
&
&
&
\bC P^i = S^{2i+1} / S^1
}$$

One observes that the composition of these maps $D_{2i} \rightarrow \bC P^i = S^{2i+1} / S^1$ is a pseudocycle representative of the fundamental class, with the omega-limit set being the image of $\partial \overline{D_{2i}} \cong S^{2n-1}$ under $\pi^{S^{2i+1}}$, which is a manifold of dimension $(2n-1)-1 = 2n-2$. We abusively denote by $\iota_i : D_{2i} \hookrightarrow S^{2i+1} \hookrightarrow S^{\infty}$. 

Given some smooth manifold $X$ of dimension $n$ equipped with an $S^1$-action, and some $S^1$-equivariant smooth map $f: X \subset A$ that is also a pseudocycle, we can define \begin{equation} \label{equation:proj-inc} \mu^f_{i}: D_{2i} \times X \xrightarrow{\iota_i \times f} E S^1 \times M \rightarrow E S^1 \times_{S^1} M,\end{equation} where the first arrow is the product of inclusion with $f$, i.e. $(v,x) \mapsto (\iota_i(v),f(x)),$ and the second map is projection. Then $\mu^f_i$ is a pseudocycle. In particular, the omega limit set is contained in the image of $$\text{Im}(\mu^f_{i-1}) \bigcup D_{2i} \times \Omega_f,$$ hence is of codimension at most $2$. This is exactly because $\partial D_{2i}$ is the $S^1$-orbit of $D_{2i-2}$ (as in \eqref{equation:d-to-s}, $\partial \overline{D_{2i}}  \setminus S^1 \cong S^1 \times D_{2i-2}$), and the omega-limit-set is naturally $S^1$-invariant due to $f$ being $S^1$-equivariant. If $f$ is an embedding of an oriented $S^1$-invariant submanifold $C$, as considered above, then we will abusively denote $\mu_i^f =: \mu_i^C$ (where importantly the terminology ``$\mu_i^C$" remembers the orientation). For brevity, in the case where $f$ is the embedding of such a submanifold $X$.

In particular, given any $S^1$-fixed closed submanifold $X \subset M$, there is some family of integral homology classes $\{ [\mu_i^X] \}_{i \ge 0} \subset H_*^{S^1}(M)$ (via Zinger \cite{zinger}).

We notice that if one replaces $X$ with $X \setminus X_0$, where $X_0 \subset X$ is of codimension $2$, then the pseudocycle associated to the inclusion of $X \setminus X_0$ represents the same homology class as that of $X$. To see this, use for example the following pseudocycle bordism: take $(X \setminus X_0) \times [0,\tfrac{1}{2}]$ and glue to $X \times [\tfrac{1}{2},1]$ along $(X \setminus X_0) \times \{ \tfrac{1}{2} \}$, then remove $X_0 \times \{ \tfrac{1}{2} \}$. This is a smooth manifold, in particular, the resulting manifold is an open submanifold of $X \times [0,1]$. Then define $F(x,t) =  f(x)$ for all $$(x,t) \in (X \setminus X_0) \times [0,\tfrac{1}{2}] \cup_{X \setminus X_0 \times \{ \tfrac{1}{2} \} } X \times [\tfrac{1}{2},1] \setminus X_0 \times \{ \tfrac{1}{2} \}.$$  This bordism then has two boundaries, and its omega-limit-set is covered by the inclusion of $X_0,$ which is of codimension $2$. We will use this fact multiple times throughout.

\subsection{Symplectic cohomology}
\label{subsec:sympl-coh}

We will later, in Section \ref{subsec:s1-eq-sympl-coh}, define $S^1$-equivariant symplectic cohomology: in particular, it will be helpful to now recall the definition of symplectic cohomology as in \cite[Section 11]{lefschetz3}, and we will apply the Borel construction as in \cite{bourgeois-oancea}.

The setup is as follows: we will restrict ourselves to the case of a symplectic Lefschetz fibration, this being the case covered in \cite{lefschetz3}. In particular, we will be looking at a symplectic fibration $p: F \rightarrow \mathbb{CP}^1$, where $F$ is a manifold without boundary of dimension $n$, and $p$ is proper. The fibre over $0$ is some manifold $M$ of dimension $n-2$. We will make the necessary assumptions from the cited paper. 

\begin{assumption}[\cite{lefschetz3}]
\label{assumption:genfibrechern}
The fibration $p$ is symplectically trivial over a small disc $B(\epsilon,0)$ around $0 \in \mathbb{CP}^1$. The class of the general fibre $[M] \in H_{n-2}(F)$ is Poincar\'e dual to $c_1(F) \in H^2(F)$.
\end{assumption}


Following is the important terminology. For any $* \in \mathbb{CP}^1$, set $M_* = f^{-1}(\{ * \})$. Fix $\dagger \in B := B(\epsilon,0)$. Denote by $E = F \setminus M$. Then $E$ is Calabi-Yau (by Assumption \ref{assumption:genfibrechern}, $c_1(E) = 0$). Note that $E \xhookrightarrow{\iota} F$ induces a pullback $\iota^*: H^*(F) \rightarrow H^*(E)$.

We thus consider symplectic cohomology on $E$, which is equipped with a proper $p: E \rightarrow \bC$ that is a locally trival symplectic fibration over an annulus around $\infty$, say $\{ b \in \bC : |b| > r_0 \}$ (the restriction of $B(\epsilon,0)$). We will fix some divisor $\Omega$ to be a representative of the symplectic class $\omega \in H^2(F)$.

\begin{properties}[Properties of $H$]
\label{properties:H}
One chooses some $H=(H_t)$ as follows: first choose a sequence of annuli $$W_j = \{ r^{-}_j \le |b| \le r^{+}_j \} \subset \{ b \in \bC : |b| > r_0 \} \setminus \bC,$$ such that $r_0< r_1^- < r_1^+ < r_2^- < \dots$, and let $W = \bigcup_i W_i$. We next choose a sequence $0<a_1 < a_2 < \dots \in \bR \setminus \bZ$ such that $\lim_j a_j = \infty$. We then choose a time-dependent Hamiltonian $H$ such that $$Dp_x(X_H) = a_j (2 \phi i p(x) \partial_{p(x)}),$$ for $x \in p^{-1}(W_j)$. We also request that, if $\mathcal{P}(H)$ is the set of $1$-periodic orbits of $H$, that any $x \in \mathcal{P}(H)$ is nondegenerate, disjoint from $\Omega$, and that if $|p(x(t))|> r^+_j$ then the loop $p(x)$ winds $>a_j$ times around $0$. This can be achieved generically in a small neighbourhood of some fixed Hamiltonian $\tilde{H}$, see \cite[Setup 11.6]{lefschetz3}.
\end{properties}

\begin{properties}
\label{properties:J}
One chooses some $J=(J_t)$ an almost complex structure on $F$ as follows: we require that each $J_t$ is compatible, and that $p$ is $J$-holomorphic on $p^{-1}(W)$. Further, we request that this almost complex structure is such that $p$ is holomorphic over $B(\epsilon,0)$: in other words, for $b \in B(\epsilon,0)$, we obtain $$J|_{ \{ b \} \times M} = J_{M_b} \oplus i,$$ for some almost complex structure $J_{M_b}$ on $M_b$ (induced by some $J_M$ on $M$). 
\end{properties}

\begin{assumption}[\cite{lefschetz3}, replacing Assumption 2.1]
The restriction of the symplectic form $\omega$ on $F$ to the general fibre $M$ (denoted $\omega_M$), and the projection of $J$ to $M$ (i.e. $J_M$), are such that, for any $J_M$-holomorphic sphere $u$ such that $\omega_M([u]) >0$, this implies $c_1(TM)([u]) \ge 0$.
\end{assumption}

This condition is called {\it nonnegatively monotone}, and we note that the assumptions we make ensure that, for the choices of compatible almost-complex-structure, $F$ is also nonnegatively monotone (i.e. it has no holomorphic spheres of negative Chern number).

One can consider $1$-periodic orbits of $H$, and assign to them the Conley-Zehnder index $i(x) \in \bZ$. 

Then define $SC^*(E,H)$ to be the graded $\bZ$-vector space with one generator, of degree $i(x)$, for each $1$-periodic orbit $x$ such that $x([0,1]) \cap M = \emptyset$. Define $$d x = \sum_y (\sum_u \epsilon(u) q^{u \cdot \Omega})y,$$ where the sum is over $u$ such that \begin{itemize}
    \item $u: \bR \times S^1 \rightarrow E$, 
    \item $\partial_s u + J_t(u)(\partial_t u - x) = 0$,
    \item $\lim_{s \rightarrow + \infty} = y, \ \lim_{s \rightarrow -\infty} = x$.
\end{itemize}

The notation $\epsilon(u)$ is a choice of orientation for $u$.

The work in \cite[Section 7]{lefschetz3} shows that this moduli space of solutions $u$ behaves nicely: in particular, that the curves form a manifold of dimension $0$, and that a different choice of data yields a well-defined continuation map. This allows us to define $$SH^*(E) = SH^*(E,H) := H^*(SC^*(E,H),d).$$

\section{LbP (localisation by pseudocycles)}
\label{sec:LbP}

\subsection{LbP for a closed manifold}
\label{subsec:lbp-closedmfd}

Suppose that $A$ is a smooth closed $m$-dimensional manifold equipped with an $S^1$-action satisfying Assumption \ref{assumption:ass1}. Recall that the fixed-point set $F$ of the $S^1$-action is a union of submanifolds $F_k$ of respective even codimension $\lambda_k$. We will fix some notation: given an oriented submanifold $Y \subset X$ of some smooth oriented manifold, the notation $\text{Bl}_X (Y)$ will refer to the oriented real blowup of $Y$ along $X$. In particular, one can view this as replacing $Y$ with its unit normal bundle: alternatively, this space is homotopic to $X$ with some tubular neighbourhood of $Y$ removed. We use the notation of a blowup as this is more fitting with the geometry of the situation. The notation $\text{Bl}^{\bC}$, when used, will refer to the complex blowup in the traditional sense.

We will first replace $A$ with $W = \text{Bl}_F(A)$, by which we mean $\text{Bl}_{F_1}(\text{Bl}_{F_2} \dots (A) \dots)$.  We define the blowdown map $BD: W \rightarrow A$. We note that there is (by assumption) an induced free $S^1$-action on $W$, and so there is a principal $S^1$-bundle $$p:W \rightarrow \widecheck{W} := W / S^1.$$ We denote by $E_{F_k}$ the exceptional divisor arising from $F_k$ (i.e. the unit normal bundle of $F_k$ in $M$), which form the disjoint boundary components of $W$. Their $S^1$-quotients $\widecheck{E}_{F_k} = E_{F_k}/ S^1$ are the disjoint boundary components of $\widecheck{W}$.

We note that there is some codimension $2$ submanifold $B \subset \widecheck{W}$ such that the restriction of $p$ to $W \setminus p^{-1}B$, denoted $$p_{-B}: W \setminus p^{-1}B \rightarrow \widecheck{W} \setminus B,$$ is trivial. To determine some choice of $B$, one can construct the associated complex line bundle $W \times_{S_1} \bC \rightarrow \widecheck{W}$, pick a generic section, and $B$ is the vanishing locus of that section. 

Once we restrict to $W \setminus p^{-1}B$, the principal $S^1$-bundle $p_{-B}$ has a section $s: \widecheck{W} \setminus B \rightarrow W \setminus p^{-1}B$. Therefore, as oriented $S^1$-manifolds we obtain: \begin{equation} \label{equation:trivialbundle} W \setminus p^{-1}B \cong S^1 \times s\left( \widecheck{W} \setminus B \right).\end{equation} Providing $s\left( \widecheck{W} \setminus B \right)$ with the induced orientation, we pull back via $s$ to yield a choice of orientation on $\widecheck{W}$.


Split $B = B_1 \sqcup \dots \sqcup B_j$ as connected components. Observe that the section $s$ can be modified to a ``section" on the oriented real blowup as follows. First, we note that a generic choice of $B$ is such that $B$ and $\widecheck{E}_{F_k}$ intersect transversely. Hence, $E_B$ and $\widecheck{E}_{F_k}$ intersect transversely. We then define $\text{Bl}_B({\widecheck{W}})^{wc}$ to be:
\begin{itemize}
    \item start with $\text{Bl}_B({\widecheck{W}})$,
    \item for each $k$ such that $\lambda_k = 2$, we remove $E_B \cap \widecheck{E}_{F_k}$,
    \item for each $k$ such that $\lambda_k > 2$, we remove $\widecheck{E}_{F_k}$.
\end{itemize}

We now show that $s$ extends to $$s_B: \text{Bl}_B({\widecheck{W}})^{wc} \rightarrow W.$$ We will abusively denote by $E_B$ the restriction of the exceptional divisor $E_B$ in $\text{Bl}_B({\widecheck{W}})$ to $\text{Bl}_B({\widecheck{W}})^{wc}$. Then we can pick $s_B$ such that $s_B(E_B) \subset p^{-1}(B)$. To see this, we first observe that $s$ restricts to a map $$s|_{\widecheck{W} \setminus \widecheck{T}}: \widecheck{W} \setminus \widecheck{T} \rightarrow W \setminus T,$$ by removing an $S^1$-equivariant tubular neighbourhood $T$ of $p^{-1}B$ in $W$, diffeomorphic to the unit normal bundle of $p^{-1}B$ in $W$, and the quotient tubular neighbourhood $\widecheck{T}$ of $B$ (by the $S^1$-equivariance) in $\widecheck{W}$. Next we observe that both $\text{Bl}_{p^{-1}B} T$ and $\text{Bl}_{B} \widecheck{T}$ are $D^1$-bundles over the exceptional divisor $E_B$.

We can decompose \begin{equation} \label{equation:decomp-blowup-1} \text{Bl}_{p^{-1}B}(W) \cong W \setminus T \bigcup_{\partial T \simeq \{1 \} \times \partial T} [0,1] \times \partial T,\end{equation}and \begin{equation}\label{equation:decomp-blowup-2} \text{Bl}_B (\widecheck{W}) \cong \widecheck{W} \setminus \widecheck{T} \bigcup_{\partial \widecheck{T} \simeq \{ 1 \} \times \partial \widecheck{T}} [0,1] \times \partial \widecheck{T}.\end{equation} This relies on the fact that the exceptional divisor is diffeomorphic to the unit normal bundle, and that a $D^1$-bundle is trivial if and only if it is oriented. We note that there is an induced section $s_T : \partial \tfrac{1}{2} \widecheck{T} \rightarrow \partial \tfrac{1}{2} T$ arising from restriction, where if we associate $T$ (resp $\widecheck{T}$) with the disc bundle over $B$, then $\tfrac{1}{2} T$ (resp $\tfrac{1}{2} \widecheck{T}$) is associated the the disc subbundle of radius $\tfrac{1}{2}$.

Finally, define the blow-down map  $BD_{p^{-1}B}: \text{Bl}_{p^{-1}B}(W) \rightarrow W$, where we collapse the fibres of the unit normal bundle over $p^{-1}B$. We now write down the following section, using the decompositions of the blowups \eqref{equation:decomp-blowup-1}, and \eqref{equation:decomp-blowup-2}, in addition to a bump function $\beta : [0,1] \rightarrow \bR$:

\begin{equation}\begin{array}{ll} s_B : & \text{Bl}_B \widecheck{W} \rightarrow W \\ s_B(x) = & \begin{cases} \begin{array}{ll} s|_{\widecheck{W} \setminus \widecheck{T}}(x) & \text{ if } x \in \widecheck{W} \setminus \widecheck{T} \\ \beta(2t-1) s|_{\widecheck{W} \setminus \tfrac{1}{2} \widecheck{T}}(t,\hat{x}) + (1-\beta(2t-1)) BD_{p^{-1}B}(t,s_T(\hat{x})) & \text{ if } x = (t,\hat{x}) \in [\tfrac{1}{2},1] \times \partial \widecheck{T}, \\ BD_{p^{-1}B}(t,s_T(\hat{x})) & \text{ if } x = (t,\hat{x}) \in [0, \tfrac{1}{2}] \times \partial \widecheck{T}.\end{array}\end{cases} \end{array} \end{equation}

\begin{definition}
\label{definition:weight}
Given some connected component of $B$, i.e. some $B_k$, pick any $b \in B_k$ and consider the copy of $S^1$ sitting over $b$ in $\text{Bl}_B(\widecheck{W})$, corresponding to the $S^1$-fibre over $b$ of the unit normal bundle, $S_B(\widecheck{W})_b$, of $B$ in $\widecheck{W}$. The section $s_B$ can be restricted to $S_B(\widecheck{W})_b \cong S^1$, and it lands in $p^{-1}(b) \cong S^1$. Let $n_k$ be the oriented weight of $s_B|_{S_b(\widecheck{W})}$, i.e. an orientation on $B$, $p^{-1}B$ and $S_B(\widecheck{W})$ induces parametrisations on $S_B(\widecheck{W})_b \cong S^1$ and $p^{-1}B \cong S^1$, from whence we can deduce an element of $\pi_1(S^1) \cong \bZ$. This is alternatively the weight as an $S^1$-representation.

\end{definition}

We now continue with our setup for LbP. We will denote by $\text{Bl}_B (\widecheck{W})$ the real oriented blow-up of $\widecheck{W}$ along $B$ (which is a manifold with codimension $2$ corners), i.e. $\text{Bl}_B (\widecheck{W}) \cong \widecheck{W} \setminus T_B$ wherein $T_B$ is some tubular neighbourhood of $B \subset \widecheck{W}$. Denote by $\text{Bl}_B (\widecheck{W})^{wc}$ the manifold-with-corners $\text{Bl}_B(\widecheck{W})$ with the following removed:

\begin{itemize}
    \item exceptional divisors $E_{F_k}$ when $\lambda_k > 2$, 
    \item the codimension $2$ corners (the intersections of the exceptional divisor $E_B$ of $\text{Bl}_B (\widecheck{W})$ with the exceptional divisors $E_{F_k}$ when $\lambda_k = 2$).
\end{itemize}

In this way, $\text{Bl}_B (\widecheck{W})^{wc}$ is a (perhaps noncompact) smooth manifold with boundary.

We abusively notate the following smooth map \begin{equation} \label{equation:pseudo-bordism} s_{B,i} : \overline{D_{2i}}\times {\text{Bl}_B({\widecheck{W}})}^{wc}) \setminus Z_i \rightarrow E S^1 \times W \xrightarrow{id \times BD} E S^1 \times A \rightarrow E S^1 \times_{S^1} A.\end{equation} The set removed from the domain $Z_i$ is the codimension $2$ corners of $\overline{D_{2i}}\times {\text{Bl}_B({\widecheck{W}})}^{wc})$, i.e. $$Z_i := \Biggl(\partial \overline{D_{2i}} \times \biggl( E_{B} \setminus E_B \cap \widecheck{E}_{F_k}\biggr)\Biggr) \bigcup \Biggl(\partial \overline{D_{2i}} \times \biggl( \bigsqcup_{k : \ \lambda_k = 2} \widecheck{E}_{F_k} \setminus E_{B} \cap \widecheck{E}_{F_k}\biggr) \Biggr).$$ In particular, the domain is a smooth manifold of dimension $2i+m-1$ with boundary. 

We briefly note the following: if $\lambda_k = 2$ then $\widecheck{E}_{F_k} = E_{F_k} / S^1 \cong F_k$, and generically $B$ intersects $E_{F_k}$ transversely, so we can denote by $F^0_k = E_{F_k} \cap B \subset F_k$ a smooth codimension $2$ submanifold (in particular, $\dim(F^0_k) = m-4$. 

We further note that $BD(p^{-1}B)\setminus F$ is an embedded submanifold of $A$ of dimension $m-2$, and further defines a pseudocycle: this is because its omega-limit-set consists of the union of $F^0_k$ for $\lambda_k = 2$, and $F_j$ such that $\lambda_j \ge 4$, hence the omega-limit-set is of dimension at most $m-4$.

We will now use \eqref{equation:pseudo-bordism} to prove LbP.


\begin{theorem}[LbP] \label{theorem:lbp} \begin{equation} \label{equation:lbp1} \sum_{k \ : \ \lambda_k = 2} \epsilon(F_k) [\mu_i^F] + \sum_k n_k [\mu_i^{BD(p^{-1}B_k)^{wc}}] - [\mu_{i-1}^A] = 0.\end{equation}\end{theorem}
\begin{proof}
  We first demonstrate that the smooth map \eqref{equation:pseudo-bordism} is a pseudocycle bordism of dimension $2i+m-2$. To compute the dimension of its omega limit set, we observe that the omega limit set is the union of four sets: two sets arising from $\text{Bl}_B (\widecheck{W}) \setminus \text{Bl}_B (\widecheck{W})^{wc}$, and two cases arising from $Z_i$. We deal with them as follows:

\begin{itemize}
\item for each $k$ such that $\lambda_k = 2$, we have removed $E_B \cap \widecheck{E}_{F_k}$ from $\text{Bl}_B (\widecheck{W})$: the associated piece of the omega-limit-set is covered by the image of $D_{2i} \times F^0_k$, where we observe that $F^0_k$ is an embedded codimension $2$ submanifold of $F_k$, which itself is of codimension $2$ in $m$. Thus, this piece of the omega limit set is of dimension at most $2i + m-4$. 
\item for each $k$ such that $\lambda_k > 2$, we have removed $\widecheck{E}_{F_k}$ from $\text{Bl}_B (\widecheck{W})$: but $\lambda_k > 2$ implies $\lambda_k \ge 4$, and so the associated image from the blown-up section applied to $\widecheck{E}_{F_k}$ is covered by the image of $D_{2i} \times F_k$, hence the associated piece of the omega-limit-set is of dimension at most $2i+m-4$.
\item we removed $\partial \overline{D_{2i}} \times E_{B} \setminus E_B \cap \widecheck{E}_{F_k}$: the image of this under $s_{B,i}$ is covered by the image of $\mu^{p^{-1}B \setminus F^0_k}_{i-1}$, where we use that $p^{-1}B$ is invariant under the $S^1$-action, and $d \overline{D_{2i}}$ is the $S^1$-orbit of $\overline{D_{2i-2}}$. This is a map from a manifold of dimension at most $2i-2+m-2 = 2i+m-4$,
\item we removed $\partial \overline{D_{2i}} \times \widecheck{E}_{F_k} \setminus E_B \cap \widecheck{E}_{F_k}$ when $\lambda_k = 2$, and the image under this is covered by $\mu^{F_k \setminus F^0_k}_{i-1}$,  where we use that $F$ is invariant under the $S^1$-action, and $d \overline{D_{2i}}$ is the $S^1$-orbit of $\overline{D_{2i-2}}$. This is a map from a manifold of dimension at most $2i-2+m-2 = 2i+m-4$.
\end{itemize}

Hence, the omega limit set is of dimension at most $2i+m-4$.

Now we wish to compute the boundaries of this pseudocycle bordism. One considers the boundaries of $\overline{D_{2i}} \times {\text{Bl}_B({\widecheck{W}})}^{wc} \setminus Z_i$, which are:

\begin{enumerate}
    \item taking the product of $D_{2i}$ with the boundary of ${\text{Bl}_B({\widecheck{W}})}^{wc}$ that corresponds to the union of $\widecheck{E_{F_k}} \setminus (\widecheck{E_{F_k}} \cap E_B) \subset \widecheck{W}$ for $k$ such that $\lambda_k = 2$,
    \item taking the product of $D_{2i}$ with the boundary of ${\text{Bl}_B({\widecheck{W}})}^{wc}$ corresponding to the exceptional divisor $E_B \setminus (E_F \cap E_B)$, which represents the unit normal bundle of $B \setminus E_F \subset {\widecheck{W}} \setminus E_F$. Because $B$ is a union of connected components $B_k$, we obtain the union of the normal bundles for each $B_k$.
    \item taking the boundary of $D_{2i}$, which is diffeomorphic to $S^{2i-1}$. Up to the removal of a codimension $2$ submanifold, i.e. $S^{2i-3}$, this is diffeomorphic to $S^1 \times D_{2i-2}$.
\end{enumerate} 

Now we want to explore (respectively) what each of these boundaries yields as pseudocycles.

\begin{enumerate}
    \item In this case, for each $k$ such that $\lambda_k =2$, the associated pseudocycle is some sign multiplied by $\mu^{F_k \setminus F^0_k}_i$. Notice that $\widecheck{E}_{F_k}$ has an induced orientation as a boundary of $\widecheck{W}$ and $F$ came equipped with a choice of orientation in order to determine homology classes $\mu^F_i$. There is a sequence of maps \begin{equation}\label{equation:F-boundary} \widecheck{E}_{F_k} \setminus \left( \widecheck{E}_{F_k} \cap E_B \right) \xrightarrow{s_B} E_{F_k} \setminus (E_{F_k} \cap p^{-1}B)  \rightarrow F_k \setminus F^0_k.\end{equation} The first map is the section $s$, and the second map is the collapsing of the fibres. One observes that this map is necessarily the identity map, there is an obvious inverse, so we can assume that this is a copy of $F_k$ embedded in $\widecheck{W}$. Observe also that, by our choice of orientation on $\widecheck{W}$, boundary component submanifold $\widecheck{E}_{F_k} \setminus (\widecheck{E}_{F_k} \cap E_B)$ is in/outwardly oriented exactly when $E_{F_k} \setminus (E_{F_k} \cap p^{-1}B)$ is in/outwardly oriented. Thus define $\epsilon(F_k) = +1$ if $E_{F_k} \setminus (E_{F_k} \cap p^{-1}B)$ is outwardly oriented (and $-1$ if inwardly). By definition, if we pre-compose $\mu^{F_k \setminus F^0_k}_i$ with the map in \eqref{equation:F-boundary}, we obtain exactly $\epsilon(F_k)$ multiplied by the pseudocycle contribution from $D_{2i} \times \widecheck{E}_{F_k} \setminus \left( \widecheck{E}_{F_k} \cap E_B \right)$. Because $F^0_k$ is of codimension $2$, $\mu^{F_k \setminus F^0_k}_i = \mu^F_i$. We will denote this pseudocycle $\epsilon(F_k) \mu_i^{F_k}$ (abusing our earlier notation by representing by $F$ the embedding of $F \rightarrow M$).
    
    \item Notice, as in the last point, that $E_{B_k}$ comes equipped with an orientation as a boundary of $\text{Bl}_{B}(\widecheck{W})^{wc}$. Notice further that we have chosen some orientation of $p^{-1}B_k$ (in order to define the class $\mu_{i}^{BD((p^{-1} B_k)^{wc})}$), and by considering the local trivialisation of the fibre bundle (and using the induced orientation on the $S^1$-action) this is the same as choosing an orientation on $B_k$. With an orientation on $B_k$ and $E_{B_k}$, we obtain an orientation locally on the $S^1$-fibres of $E_{B_k} \rightarrow B_k$. Further, we notice that we can write the following map of bundles:
    $$\xymatrix{
E_{B_k}
\ar@{->}^-{s_B}[r]
\ar@{->}_-{}[d]
&
p^{-1} B_k
\ar@{->}^-{}[d]
\\
B_k
\ar@{->}^-{=}[r]
&
B_k
}$$ which is locally, on a trivialising neighbourhood $U \subset B_k$, of the form $n_k \times \text{Id}: S^1 \times U \rightarrow S^1 \times U$, by definition.
Further, because of the way in which we extended our section $s$ to $s_{B}$, we note that $s_{B,i}|_{D_{2i} \times E_{B_k}}$ has the same intersection number with any other pseudocycle as $n_k$ multiplied by the intersection number with $D_{2k} \times BD((p^{-1} B_k)^{wc})$. In particular, observing that we can consider the restriction $$s_{B_k,i}|_{D_{2i} \times E_{B_k}}: D_{2i} \times E_{B_k} \rightarrow D_{2i} \times BD(p^{-1}B_k)^{wc},$$ we observe that $$\left(s_{B_i}|_{D_{2i} \times E_{B_k}}\right)_*[D_{2i} \times E_{B_k}] = n_k [D_{2i} \times BD(p^{-1}B_k)^{wc}] \in H_*^{lf}(S^{2i+1} \times_{S^1} BD(p^{-1}B_k)^{wc}),$$ the classes $[-]$ being respectively the locally finite fundamental classes of $D_{2i} \times E_{B_k}$ and $D_{2i} \times BD(p^{-1}B_k)^{wc}$ (using \cite{lf-pseudocycles}). This implies that the pseudocycle associated to the boundary is $n_k \cdot \mu_{i}^{BD((p^{-1} B_k)^{wc})}$. We notice that the omega limit set is covered by the images of the $F_j$ such that $\lambda_j > 2$, and the $F^0_k$ such that $\lambda_k = 2$. Further, $F^0_k$ is codimension $2$ in $p^{-1}B_k$ for all $k$. 
    
    \item For the map $$s_{B,i}|_{S^1 \times D_{2i-2} \times \text{Bl}_B({\widecheck{W}})}^{wc} : S^1 \times D_{2i-2} \times \text{Bl}_B({\widecheck{W}})^{wc} \rightarrow E S^1 \times_{S^1} A,$$ we recall that we may remove submanifolds in the domain such that their image is of dimension at most $1+2i-2 + m-1 - 2 =2i+m-4$: by removing the $S^1 \times D_{2i-2} \times \widecheck{E}_{F_k}$ for $\lambda_k > 2$, and $S^1 \times D_{2i-2} \times E_B$, we observe that this pseudocycle is equivalent to the map $$S^1 \times D_{2i-2} \times \left( (A \setminus F)/S^1 \right) \setminus B) \rightarrow  ES^1 \times_{S^1} A.$$ By pre-swapping the first two terms (which preserves the orientation as $D_{2i-2}$ is of even dimension), we obtain \begin{equation} \label{equation:A-pseudocycle-bordism-boundary} D_{2i-2} \times S^1 \times \left( A \setminus F \setminus p^{-1}B \right) / S^1 \xrightarrow{\cong} S^1 \times D_{2i-2} \times \left( (A \setminus F)/S^1 \right) \setminus B) \xrightarrow{s_{B,i}} E S^1 \times_{S^1} A.\end{equation} Next recall that the $S^1$-action was diagonal, and so the $S^1 \times \left( A \setminus F \setminus p^{-1}B \right) / S^1$ factor on the left-hand side of \eqref{equation:A-pseudocycle-bordism-boundary} is oriented negatively with respect to Equation \eqref{equation:trivialbundle} (alternatively, it is \eqref{equation:trivialbundle} with the antipodal $S^1$-action). So we must finally pre-compose with the orientation-reversing diffeomorphism that is $id \times \text{antipodal map} \times id$. We thus notice that this is exactly the pseudocycle representing the negative of the $\mu_{i-1}^{(A \setminus F) \setminus p^{-1}B}$. Hence, we obtain the pseudocycle $-\mu_{i-1}^{A}$, again abusively denoting by $A$ the inclusion $A \rightarrow A$, which by our choice of orientation is obtained with coefficient $1$. 
\end{enumerate}

These pseudocycles above thus correspond respectively to the first, second and third terms in \eqref{equation:lbp1}.
\end{proof}

\begin{remark}
    
The $S^1$-localisation we state here is an abridged form of Theorem 3.3 in \cite{atiyahbott}. The relevant theorem is the following: 

\begin{theorem*}[Atiyah-Bott, \cite{atiyahbott}]
\label{theorem:s1-localisation}
Given a smooth action of $S^1$ acting on a smooth manifold $X$, denote by $i: F \rightarrow X$ the inclusion of the fixed point set, the map $$i^* : H^*_{S^1}(X) \rightarrow H^*_{S^1}(F),$$ has kernels and cokernels that are $u$-torsion. 

\end{theorem*}

It is important to draw parallels with our above result. The idea of our result is that, given an element of $H_*^{S^1}(X)$, by increasing the equivariant parameter space (i.e. considering $D_{2i+2}$ instead of $D_{2i}$), one may ``decrease the dimension" of the submanifold of $X$ that we consider. Doing this sufficiently many times, one intuitively expects to ``isolate" the fixed-point set $F$ (subject to some characteristic classes of the normal bundle).

We can see  that $F^0_k \subset p^{-1} B$ is related to the Euler class of the normal bundle of $F$ in $M$, which is not dissimilar to e.g. the Atiyah-Bott localisation formula. Future work will explore this link further.

\end{remark}

\subsection{LbP for a manifold with boundary}
\label{subsec:lbp-mfdwbdry}

We will extend the results of the previous section to the case where $A$ is a smooth compact $n$-dimensional manifold with boundary $N = \partial A$, although we will make the following extra assumption:

\begin{assumption}
\label{assumption:mfdwbdry-assumption}

The $S^1$-action acts freely on $N$, and for a generic choice of $B$ such that $W \setminus p^{-1}B \rightarrow \widecheck{W} \setminus B$ is trivial, then $\partial B := N/S^1 \cap B$ is transverse, and the map $p^{-1}(\partial B) \rightarrow \partial B$ is a trivial $S^1$-bundle. 

\end{assumption}

This assumption can be unnecessary, but circumventing it adds extra technical headaches regarding nested iteration, and it will not be necessary for the application in Section \ref{sec:lifting-to-moduli-spaces} of this paper. In particular, without the freeness condition on $N$, one would have to apply LbP to $N$ additionally. The triviality of the $S^1$-bundle $p^{-1}(\partial B) \rightarrow \partial B$ allows us to consider $\partial B$ as a submanifold of $N$, as opposed to being forced to iteratively remove a codimension $2$ submanifold of $\partial B$ and so forth. Assumption \ref{assumption:mfdwbdry-assumption} serves to simplify the eventual result, Corollary \ref{corollary:lbp-mfd-w-bdry}.

It is important here to note that the work in \cite{zinger} does not immediately apply to relative homology for manifolds with boundaries (although it should be related to \cite{lf-pseudocycles}). However, we can avoid having to prove this technical underpinning in the following way: replace the manifold-with-boundary $A$ with $A \cup_{N} N \times [0,\epsilon)$, thickening the boundary (here is is important that $A$ is orientable). Then $N$ becomes an embedded smooth submanifold of $M$, and one can for example define pseudocycles as below (their image being contained in $A$, hence precompact).

\subsubsection{Gluing to get a pseudocycle for a submanifold with boundary}
In this situation, we notice the following: we can still obtain a codimension $2$ subset $B \subset (A \setminus F)/S^1$, the removal of which trivialises our $S^1$ bundle. Generically it satisfies $\partial B = B \pitchfork (N/S^1)$. The first part of Assumption \ref{assumption:mfdwbdry-assumption} assumes that $F$ misses $N$. We note first that $\mu^A_i$ is not an embedded closed submanifold, but rather an embedded bordism in $S^{2i+1} \times_{S^1} A$. In particular, it has a boundary $\mu^{\partial A}_i = \mu^N_i$. For later use, we will actually think of $\mu^A_i$ as a pseudocycle bordism in $S^{2i+3} \times_{S^1} A$.

We note the following: $N / S^1 \subset (A \setminus F) / S^1$ is some smooth submanifold. Hence, if we remove $B$ then we obtain $$\left( N / S^1 \right) \setminus B \subset \left((A \setminus F) / S^1 \right) \setminus B \xrightarrow{s} A \setminus (p^{-1}B \cup F).$$ In particular,  $\mu^{s((N / S^1) \setminus B)}_{i+1} = (\text{id} \times s) \circ \mu^{(N /S^1) \setminus \partial B}_{i+1}$ is a pseudocycle in $S^{2i+3} \times_{S^1} A$, and its boundary is $$\partial D_{2i+2} \times s((N/S^1 ) \setminus \partial B) \sim S^1 \times D_{2i} \times s((N/S^1 ) \setminus \partial B) \sim D_{2i} \times N,$$ where we use $\sim$ to denote ``diffeomorphic up to removing a codimension $2$ submanifold". This means that the domains of $\mu^{s(N / S^1 \setminus B}_{i+1}$ and $\mu^{A}_i$ have boundaries that are diffeomorphic, related by an orientation reversing diffeomorphism (recalling as in \eqref{equation:A-pseudocycle-bordism-boundary} that the orientation is swapped because the $S^1$-action is diagonal). Further, it means that (after passing to the complement of some codimension $2$ submanifold) we can construct a continuous map $\mu^{glued}_i$ by gluing together $\mu^{s((N / S^1) \setminus B}_{i+1}$ and $\mu^{A}_i$, such that the domain of this glued map is the smooth manifold obtained by attaching collar neighbourhoods of $\partial D_{2i+2} \times s((N/S^1) \setminus \partial B)$ to $D_{2i} \times \partial A$ via the diffeomorphism: \begin{equation} \label{equation:boundaries-boundaries} \begin{array}{ll} \partial (D_{2i+2} \setminus \partial D_{2i}) \times ((N / S^1) \setminus \partial B)) & \cong  \partial (D_{2i+2} \setminus \partial D_{2i}) \times ((N / S^1) \setminus \partial B) \\& \cong (S^1 \times D_{2i}) \times ((N / S^1) \setminus \partial B) \\ & \cong  D_{2i} \times (S^1 \times \left((N / S^1) \setminus \partial B\right) \\ & \cong D_{2i} \times (N \setminus p^{-1}\partial B)) \\ & \cong D_{2i} \times \partial (A \setminus p^{-1}B).\end{array} \end{equation} We call this glued manifold \begin{equation} \label{equation:boundary-glued-manifold} D_{2i+2} \times s((N/S^1) \setminus \partial B) \# D_{2i} \times A.\end{equation}

Importantly, $$\mu^{glued}_i : D_{2i+2} \times s((N/S^1) \setminus \partial B) \# D_{2i} \times A \rightarrow S^{2i+3} \times_{S^1} A$$ is, in general, not smooth. Indeed, it is non-smooth specifically along the image of $D_{2i} \times N$: we can see this with local coordinates. Suppose near $$\mu^{glued}_i(D_{2i} \times N) \subset \text{Im}(\mu^{glued}_i) \subset \overline{D_{2i+2}} \times A \rightarrow S^{2i+3} \times_{S^1} A$$ we use local coordinates $(x_1,x_2,x_3, x_4, x_5)$ where:
\begin{enumerate}
    \item $x_1 \in \mathbb{R}^{2i}$ parametrises a ball in $D_{2i}$,
    \item $x_2 \in [0,\epsilon)$ parametrises the normal direction of the boundary $D_{2i} \times S^1 \subset \partial D_{2i+2} \subset D_{2i+2}$, 
    \item $x_3$ parametrises some interval of $S^1$,
    \item $x_4 \in (-\epsilon,0]$ parametrises some normal direction of $N/S^1 \subset A / S^1$,
    \item $x_5 \in \mathbb{R}^{2n-2}$ parametrises $s(N / S^1 \setminus B) \subset A$.
\end{enumerate}  

In particular, $\mu^{glued}_i(D_{2i} \times N)$ is obtained when $x_2 = x_4 = 0$.

Then if we were to assume (for contradiction) that $\mu^{glued}_i$ were smooth, approaching $\partial (D_{2i+2} \setminus \partial D_{2i}) \times (N / S^1 \setminus \partial B))$ along $D_{2i+2} \times s(N/S^1 \setminus \partial B)$ we observe that \begin{equation} \label{equation:mu-glued-1} \tfrac{d\mu^{glued}_i}{d x_2} \biggr|_{x_2 = x_4 = 0} \neq 0, \ \tfrac{d \mu^{glued}_i}{d x_2} \biggr|_{x_2 = x_4 = 0} =0.\end{equation} However, approaching $\cong D_{2i} \times \partial (A \setminus p^{-1}B)$ along $D_{2i} \times (A \setminus p^{-1}B)$ we obtain that \begin{equation} \label{equation:mu-glued-2} \tfrac{d\mu^{glued}_i}{d x_2} \biggr|_{x_2 = x_4 = 0} =0, \ \tfrac{d \mu^{glued}_i}{d x_2} \biggr|_{x_2 = x_4 = 0} \neq 0,\end{equation} thus providing our contradiction.

This being the case, we need to smooth our $\mu^{glued}_i$. The only disparity in smoothness arises from  \eqref{equation:mu-glued-1} and \eqref{equation:mu-glued-2}, and we can fix this by applying a smoothing exactly as one would smooth the inclusion of $\{(x,0) : x \ge 0 \} \cup \{ (0,y):  y \ge 0 \} \subset \mathbb{R}^2$ in some small ball around $\{ (0,0) \}$. In particular, there is a smoothing parameter $s=s(x_1,x_3,x_5) \in (0,\delta)$ (i.e. a smooth function of $x_1,x_3,x_5$) corresponding to smoothings of $\mu^{glued}_i$ that are of $C^{\infty}$-distance $s$ from $\mu^{glued}_i$. The only extra consideration is that we want to preserve the omega-limit-set (i.e. ensure it does not increase in dimension). To do this, we observe that the omega-limit-set of $\mu^{glued}_i$ is some union of submanifolds of $S^{2i+3} \times_{S^1} A$. We thus require the smoothing of $\mu^{glued}_i$ to converges to $\mu^{glued}_i$ as one approaches the omega-limit-set. We can do this by taking a sequence of nested punctured neighbourhoods $T_1 \supset T_2 \supset \dots \supset \Omega_{\mu^{glued}_i}$, such that $\bigcap_k T_k = \Omega_{\mu^{glued}_i}$. Using local coordinates $x_{i,k}$ for $i=1,2,3,4,5$ being the restriction of $x_i$ to $T_k$, the condition that we require is: \begin{equation}\text{sup}_{(x_{1,k}, x_{3,k}, x_{5,k})} s(x_{1,k}, x_{3,k}, x_{5,k}) \rightarrow 0 \text{ as } k \rightarrow \infty.\end{equation} Hence, for any $s$ there is an associated glued pseudocycle, and any two such $s, s'$ provide bordant pseudocycles, using the bordism arising from an interpolation $s + \lambda(s'-s)$. In particular, for any such choice of $s$ we denote the associated pseudocycle by $\mu^{A \# N/S^1}_i$.

Depending on whether $\partial B_k = \emptyset$, we may either immediately obtain a pseudocycle $\mu^{p^{-1}B_k}_{i}$ or we may have to glue as with the $\mu^A_i$. By Assumption \ref{assumption:mfdwbdry-assumption} there is a section $\partial B \rightarrow p^{-1}\partial B$, and hence there is a well defined inclusion map associated to this section that defines a pseudocycle (perhaps bordism) $\mu^{\partial B}_{i+1}$. Further, as above the boundary matches that of $\mu^{p^{-1}B}_i$, hence we obtain a pseudocycle bordism $\mu^{p^{-1}B_k}_i \# \mu^{\partial B_k}_{i+1}$ by once again gluing and smoothing. See Remark \ref{remark:assumption-unrealistic} for a discussion of Assumption \ref{assumption:mfdwbdry-assumption} in practise.

If, as previously, we denote by $\epsilon(F_k)$ the sign difference in the orientation of $F_k \setminus F^0_k \subset A$ and $\rho_k \left( (E_{F_k} \setminus E_{F_k}|_{F^0_k}) / S^1 \right)$, where $\rho_k: E_{F_k} \rightarrow F_k$ is the induced projection associated to the blow-down, then we need to make some small but important changes (see the proof of Corollary \ref{corollary:lbp-mfd-w-bdry}) so that everything follows as previously. In particular we obtain that:

\begin{corollary}[LbP with boundary] \label{corollary:lbp-mfd-w-bdry} \begin{equation} \label{equation:lbp2} \sum_{k \ : \ \lambda_k = 2}  \epsilon(F_k) [\mu_i^{F_k}] + \sum_k n_k [\mu_i^{(p^{-1}B_k)^{wc}} \# \mu_{i+1}^{\partial B_k}] - [\mu^A_{i-1} \#_{\mu^N_{i-1}} \mu^{N/S^1}_{i}] = 0.\end{equation}\end{corollary}

\begin{proof}
We construct a pseudocycle bordism exctly as in the proof of Theorem \ref{theorem:lbp}. Each of the existing three types of boundary components remain. However, the pseudocycle bordism $s_{B,i}$ of dimension $\dim A -1 + 2i$ from Equation \ref{equation:pseudo-bordism} now has an extra boundary component. This component arises from the boundary of $\text{Bl}_B(\widecheck{W})^{wc}$ induced by $\partial A= N$ quotiented by $S^1$ but then blown-up along $B$. In other words, this part of the boundary is diffeomorphic to $$D_{2i} \times \left(N / S^1\right) \setminus \left(T_{\partial B} N/S^1\right) \cong D_{2i} \times \text{Bl}_{\partial B} (N / S^1),$$ where $T_{\partial B} N/S^1$ is a tubular neighbourhood of $\partial B \subset N/S^1$. 

If one removes some codimension $2$ submanifold $B' \subset \partial B$, then $T_{\partial B \setminus B'} N / S^1 \rightarrow \partial B \setminus B'$ is a trivial $D^2$-bundle, i.e. $T_{\partial B \setminus B'} N / S^1 \cong (D^2 \times \partial B \setminus B')$. In particular, after removing a codimension $2$ submanifold, this extra boundary component (which we establised above is diffeomorphic to $D_{2i} \times \text{Bl}_{\partial B} (N / S^1)$) is diffeomorphic to the union of $$D_{2i} \times N / S^1$$ and a negatively oriented copy of $$D_{2i} \times D^2 \times \partial B \setminus B' \cong D_{2i+2} \times \partial B \setminus B'.$$ 

Further, due to the discussion on Equation \eqref{equation:A-pseudocycle-bordism-boundary}, we observe that this new boundary $D_{2i} \times N / S^1$ (which is positively oriented) when taken in combination with the $S^1 \times D_{2i-2} \times \widecheck{W}$ boundary contribution, yields the smooth manifold denoted $D_{2i+2} \times s((N/S^1) \setminus \partial B) \# D_{2i} \times A$ in \eqref{equation:boundary-glued-manifold}. After gluing and smoothing locally, the restriction of the pseudocycle bordism to $D_{2i+2} \times s((N/S^1) \setminus \partial B) \# D_{2i} \times A$ restricts to $-\mu^A_{i-1} \#_{\mu^N_{i-1}} \mu^{N/S^1}_{i}$. 

The same holds similarly for $D_{2i+2} \times \partial B \setminus B'$ being glued to $D_{2i} \times E_{B_k} \setminus (E_{B_k} \cap \widecheck{E}_{F})$ to obtain that the pseudocycle bordism restricts to $\mu_i^{(p^{-1}B_k)^{wc}} \# \mu_{i+1}^{\partial B_k}$ on this boundary component: in this case, one can glue in a ``constant bordism" $D_{2i+2} \times [0,1] \times \partial B \setminus B'$ (induced by the section of $p^{-1}B \rightarrow B$ for all points in $[0,1]$).

In particular, the boundary components of the pseudocycle bordism are indeed those arising in Equation \eqref{equation:lbp2}.
\end{proof}

\begin{remark}

\label{remark:assumption-unrealistic}

In Assumption \ref{assumption:mfdwbdry-assumption}, we assumed that $p^{-1}(\partial B) \rightarrow \partial B$ is a trivial $S^1$-bundle. This is in general quite a restrictive assumption, but it is enough for the examples that we find in Sections \ref{subsec:statement1} and \ref{subsec:statement2}. Further, it is straightforward to remove this assumption, but the statement of Corollary \ref{corollary:lbp-mfd-w-bdry} will grow more complicated. In essence, if $p^{-1}(\partial B) \rightarrow \partial B$ is not trivial then we must find a codimension $2$ subset of $\partial B$ that we can remove in order to trivialise. This would then be applied iteratively: our assumption implies we do not need to iterate more than once.

\end{remark}

\begin{remark}
In the case we will consider in Section \ref{sec:ourexample}, our space of domains will, after compactification, look like the closed $2$-disc. However, we will demonstrate that contributions from the boundary do not arise, and hence $\mu^A_i$ in that instance is enough to define a moduli space, even though it is not by itself a pseudocycle. For more general symplectic manifolds, one might expect contributions from an ``$S^1$-equivariant cap-product".
\end{remark}

\section{Lifting to moduli spaces}
\label{sec:lifting-to-moduli-spaces}
Note that a lot of the notation that was used in Section \ref{sec:LbP} will be used in a different context henceforth, so as to link to the discussion in \cite{lefschetz3}, in particular $F$, $E$ and $M$.

The problem in question is thus: we would like to apply the results of Section \ref{sec:LbP} directly to some moduli space $$\mathcal{M}_i(\text{data}),$$ where $\mathcal{M}_i(\text{data})$ is some moduli space parametrised by $S^{2i+1}/S^1 \subset BS^1$. These moduli spaces ``look like" strata of a homotopy quotient, but they are not in general: they will depend on a Hamlitonian perturbation and an almost complex structure $H^{S^1},J^{S^1}$ that, due to regularity issues, may not be chosen to be $S^1$-invariant. The best we can do is to require that our moduli spaces fibre over some stratum of the homotopy quotient of a smooth manifold (one can think about this as the case where there is not an $S^1$-action on a space, but rather a ``$S^1$-action up to homotopy".). 

In this section we will demonstrate through Theorem \ref{theorem:lift-to-moduli} that this is sometimes enough: in particular, given some parameter space $A$ such that $\mathcal{M}_i(A) \rightarrow S^{2i+1} \times_{S^1} A$ for each $i$, we can lift LbP from $S^{2i+1} \times_{S^1} A$ to $\mathcal{M}_i(A)$, to obtain a version of $S^1$-localisation for the moduli space itself. The idea is to construct a $1$-dimensional moduli space parametrised by the pseudocycle bordism we used in LbP, Theorem \ref{theorem:lbp}. As a general rule-of-thumb, if one wants to apply this result to moduli spaces, one first writes down the space of underlying domains, then one applies LbP to obtain a pseudocycle bordism, then one defines a $1$-dimensional moduli space parametrised by this bordism.

\subsection{$S^1$-equivariant symplectic cohomology}
\label{subsec:s1-eq-sympl-coh}

Note that $E$, as defined in Section \ref{subsec:sympl-coh}, satisfies the necessary conditions to apply the construction of symplectic cohomology from \cite[Section 11]{lefschetz3}. We will extend this definition of symplectic cohomology (recalled in Section \ref{subsec:sympl-coh}) to an $S^1$-equivariant version, using the Borel construction as in for example \cite{bourgeois-oancea} and \cite{seidelsmith}.

First, we fix as in the nonequivariant case some $H=(H_t)_{t \in S^1}$ and $J=(J_t)_{t \in S^1}$. We then fix some family of Hamiltonians $H^{eq}_{t,v}: E \rightarrow \bR$ satisfying Properties \ref{properties:H} for every $v \in S^{\infty}$, so we obtain some smooth $H^{eq}: S^1 \times E \times S^{\infty} \rightarrow \mathbb{R}$ as a $1$-periodic Hamiltonian parametrised in an $S^1$-equivariant way by the classifying space $S^{\infty} \subset \mathbb{C}^{\infty}$. We require that for any $\theta \in S^1$, $$H^{eq}(\psi,e,v) = H^{eq}(\psi + \theta, e, \theta \cdot v).$$ We sometimes denote $$H^{eq}_v := H^{eq}(-,-,v): S^1 \times E \rightarrow \bR.$$


Further, we choose some domain-dependent almost complex structure $J^{eq}_{v,t}$ on $E$, dependent on the $v \in S^{\infty}$ and $t \in S^1$ such that Property \ref{properties:J} is satisfied for all $v$, and such that for any $\theta \in S^1$, $$J^{eq}_{\theta \cdot v, t - \theta} = J^{eq}_{v,t}.$$

We require that $H^{eq}_{t,\sigma v} = H^{eq}_{t,v}$ and $J^{eq}_{t, \sigma v} = J^{eq}_{t,v}$ where $$\sigma: S^{\infty} \rightarrow S^{\infty}, \ \sigma(z_1,z_2,\dots) = (0,z_1,z_2,\dots).$$

We recall that there is a choice of $S^1$-invariant Morse-Bott function $g: S^{\infty} \rightarrow \bR$, with $$g(z_1,z_2,\dots) = \sum_{k=1}^{\infty} k |z_k|^2,$$ such that there is a sequence of critical $S^1$-orbits $c_1,c_2,\dots$ with $|c_i| = 2i$ for each $i$. For brevity we denote the set of all negative gradient flowlines of $g$ from $c_j$ to $c_k$ to be $\mathcal{F}(j,k)$. Note by the $\sigma$-invariance, that $\mathcal{F}(j,k) \cong \mathcal{F}(j-1,k-1)$ for all $j,k$.

We denote by $\mathcal{P}(H)$ the set of all pairs $(c,\gamma)$ such that $c$ is a critical orbit and $\gamma$ is a $1$-periodic Hamiltonian loop with respect to $H^{eq}_{c(t)}$. As in \cite[section 2.1]{bourgeois-oancea}, these generically come in isolated $S^1$-families. We will denote by $S_{k,\gamma}$ the $S^1$-orbit of $(c_k,\gamma)$.

Define $$SC^{eq}_*(E) := \Pi_{(c_k,\gamma) \in \mathcal{P}(H)} \mathbb{Z} \langle S_{k,\gamma} \rangle,$$ with $d^{eq}: SC^{eq}_*(E) \rightarrow SC^{eq}_{*-1}(E)$ defined such that (for $|x_{-}| - |x_{+}| = 1$) the coefficient of $(c_l,x_{-})$ in $d^{eq}((c_k, x_{+}))$ is the number of (oriented) $\mathbb{R}$-families of pairs $(w,u)$ such that $w : \mathbb{R} \rightarrow S^{\infty}$ is a negative gradient flowline of $g$ from $c_l$ to $c_k$, and a map $u: S^1 \times \mathbb{R} \rightarrow E$ such that $\partial_s u + J^{eq}_{t,w(s)}(u)(\partial_t u - X_{H^{eq}_{t,w(s)}}) = 0$, and asymptotically $\lim_{s \rightarrow \pm \infty} u(t,s) = x_{\pm}(t)$. We denote the space of such objects \begin{equation}
    \label{equation:eq-diff} \mathcal{M}(k,x_{+},l,x_{-}).
\end{equation}

We see that as we requested that our data is invariant under the shift map $\sigma$, then we can associate $$SC^{eq}_*(M) = SC_*(M)[[u]],$$ and $d^{eq} = d^0 + u d^1 + u^2 d^2 + \dots$, where $d^i(u^j \gamma)$ considers pairs as above such that $$\sigma^j w \in \mathcal{F}(i+j,j) \leftrightarrow w \in \mathcal{F}(i,0).$$ We will use this formulation in general.

\begin{remark}[The relationship between Section \ref{subsec:s1-eq-sympl-coh} and Section \ref{sec:LbP}] 
An important point is that in Section \ref{sec:LbP}, there is a different model of equivariant cohomology compared to Section \ref{subsec:s1-eq-sympl-coh}. The difference is broadly the difference between cellular and Morse cohomology. Relating them is a subtle point, and we will be careful, but the intuition is as follows.

We recall that within $S^{\infty}$, there are a choice of discs $D_{2i}$ such that the inclusion of $D_{2i}$ into $S^{\infty}$, then the quotient to $BS^1$, represents the generator of $H_{2i}(BS^1; \mathbb{Z})$. Further, one notices that $D_{2i} \cap c_{i} = \{ 0 \}$ is the centrepoint of the disc, and the intersection $F(i,j) \cap D_{2i}$ is a line segment from $0 \in D_{2i}$ to some point on the boundary.

Consider the usual isomorphism between Morse homology, and cellular homology built from Morse theory, on a manifold: given a Morse function $f$, the map $CM_*(M,f) \rightarrow C_*^{\text{cell}}(M)$ is induced by $c \in \text{crit}(f) \mapsto W^u(c,f)$. More generally, for a general decomposition by smooth cells, the map $\psi: CM_*(M,f) \rightarrow C_*^{\text{cell}}(M)$ is built such that the coefficient of $C \in \psi(c)$ is the number of $\gamma: (-\infty,0] \rightarrow M$ such that $\gamma(0) \in C$. This is likewise how one defines the map for pseudocycles. In order to avoid the technical mess of trying to define $S^1$-equivariant cohomology in this way, we have continued with the previously considered, well-known definitions. Indeed, it would require the construction of a lot of technology, for limited gain (the differentials would have to look like continuation maps, supported in $[-r+\lambda,\lambda+r]$ where $\lambda \in \mathbb{R}$ is permitted to vary).
\end{remark}

\subsection{Lifting LbP to moduli spaces}
\label{subsec:lifting-LbP}

We will now describe the procedure to lift LbP to the realm of moduli spaces, for the specific case as given at the beginning of Section \ref{sec:lifting-to-moduli-spaces}. Note that a similar result should hold in other generalities, as long as care is given to regularity and so forth, but in this paper we are using symplectic cohomology in this specific setting. 

Suppose that $E \subset F$ are symplectic manifolds satisfying conditions as in Section \ref{sec:lifting-to-moduli-spaces}. Suppose that:

\begin{itemize}
    \item we fix some nodal holomorphic curve $R$ consisting of a holomorphic thimble $T$ with a bubble tree attached at $\infty$, with the bubble tree consisting of $l$ components $R_{1},\dots,R_{l}$,
    \item we fix some $m \ge 1$. Let $A$ consist of the compactified set of tuples of distinct marked points $z_1,\dots,z_m$ in the interior of $R$, disjoint from the nodes, up to reparametrisation. For any $a \in A$ there is an underlying nodal Riemann surface $\text{Surf}(a)$ obtained by forgetting the marked points ($\text{Surf}(a)$ will consist of $R$ along with some attached bubble trees from marked points colliding). We ask that $(\text{Surf}(a),a)$ is a stable nodal holomorphic curve, for each $a \in A$. Recall that $A$ is a smooth compact complex manifold of complex dimension $m-l-1$ with corners.
    \item $R$ is equipped with a Chern signature $\alpha_R$: given some nodal holomorphic curve $C$ consisting of a holomorphic thimble with some number of bubble trees attached, the component spheres of the bubble trees indexed by some set $J$, we say that $C$ {\it has Chern signature $\alpha_C$} if for each component $C_j$ of the nodal curve $C$, there is an associated nonnegative integer $\alpha_C(C_j)$. Denote $\lambda = \alpha_C(T) + \sum_{j \in J} \alpha_C(C_j)$, the Chern number of $C$. 
    \item if $C',C$ are equipped with Chern signatures, and there is a collapsing map $$\text{coll}: \{ C_j: j \in J \} \rightarrow \{ C'_i : i \in I \} \sqcup T,$$ then we say that {\it the Chern signature $\alpha_C$ of $C$ is compatible with $\alpha_{C'}$ of $C'$} if $$\alpha_{C'}(C'_i) = \sum_{j \in J : \text{coll}(C_j) = C'_i} \alpha_C(C_j),$$ for all $i \in I$, and $$\alpha_{C'}(T) = \alpha_C(T) + \sum_{j \in J : \text{coll}(C_j) = T} \alpha_C(C_j).$$
    \item there is an induced $S^1$-action on $R$, acting by rotation on $T$ and fixing $R \setminus T$. This satisfies Assumption \ref{assumption:ass1}, and \ref{assumption:mfdwbdry-assumption},
    \item a choice of $\dagger_1,\dots,\dagger_q$ in a neighbourhood of $0 \in \bC P^1$ over which $f: F \rightarrow \bC P^1$ is trivial, and an assignment $\eta: \{ z_1,\dots,z_m \} \rightarrow \{ 0,\dagger_1,\dots,\dagger_q \}$ such that: $\# (\eta^{-1}(0) \cap C_j) = \alpha_R(C_j)$ and if $\# (\eta^{-1}(0) \cap T) \neq 0$ then there is some $\kappa =1,\dots,q$ such that $ \# \eta^{-1}(0) \le \#  \eta^{-1}(\dagger_{\kappa})$. This condition ensures the invariants we describe are closed, \i.e. in Lemma \ref{lemma:moduli-compactification}.
    \item we make some generic choice of data $H,J$ depending on $v \in S^{\infty}, \ a \in A ,\ \text{ and } z \in R$  such that $H,J$ are $S^1$-equivariant in the sense that $H_{v,a,z} = H_{\theta \dot v, \theta \cdot a, \theta \cdot z}$  and similarly for $J$, and they are independent of $v$ when $z$ is near the cylindrical end of $T$, and they are independent of $v$ in some neighbourhood of each marked-point and node.  
    \item for a generic choice of $H,J$ (subject to conditions above) and for some choices of pseudocycle bordism $G$ and pseudocycles $g_1,g_2$ (such that $\partial G = g_1 \sqcup -g_2$) in $S^{2i+1} \times_{S^1} A$, and for $y \in \mathcal{P}(H)$, if $$\dim(G) -|y| - 2\lambda - \sum_{k=1}^m (n-\text{dim}(f_k)) = 1,$$ then one can define smooth moduli spaces (perhaps with boundary) $$\mathcal{M}_{f}(y,\lambda), \ \text{ for } f = G,g_1,g_2$$ of dimension $\dim(f) -|y| - 2 \lambda - \sum_{k=1}^m (n-\text{dim}(f_k))$ consisting of triples $(a,w,u)$ such that:
\begin{enumerate}
    \item $w \in \mathcal{F}(i,0)$,
    \item $a \in A$,
    \item $[w(0),a] \in \text{Im}(f)$,
    \item $u : \text{Surf}(a) \rightarrow F$,
    \item For $s,t \in [1,\infty) \times S^1$ on the positive cylindrical end of $\text{Surf}(a)$, $$\lim_{s \rightarrow \infty} u(s,t) = y(t).$$
    \item for each $z \in \text{Surf}(a)$, $$\begin{cases}\begin{array}{l} du|_z \circ j|_z = J_{a,z}|_{u(z)} \circ (du|_z - X_{H_{a,z}}) \text{ for } z \text{ on } R_k \text{ for any } k, \\ du|_z \circ j|_z = J_{w(-\text{log}|z|),a,z}|_{u(z)} \circ (du|_z - X_{H_{w(-\text{log}|z|),a,z}}) \text{ for } z \neq \infty \text{ on } T \cong \left( \bC \cup \{ \infty \} \right) \setminus \{ 0 \}  \end{array}\end{cases}$$ 
    \item the induced Chern signature $$\alpha_{\text{Surf}(a)}(\text{Surf}(a)_{j}) := u|_{\text{Surf}(a)_{j}} \cdot M$$ for each component of $\text{Surf}(a)$ is compatible with $\alpha_R$ under the natural collapse map,
    \item $u(z_k) \in M_{\eta(z_k)}$ for $k=1,\dots,m$.
\end{enumerate}

\end{itemize}

\begin{lemma}
\label{lemma:moduli-compactification}
When $\dim(f) -|y| - 2\lambda - \sum_{k=1}^m (n-\text{dim}(f_k)) = 0$ the moduli space $\mathcal{M}_{f}(y,\lambda)$ is compact. When $\dim(f) -|y| - 2\lambda - \sum_{k=1}^m (n-\text{dim}(f_k))= 1$ the moduli space $\mathcal{M}_{f}(y,j)$ can be compactified by adding broken configurations such that (when $w$ is a parametrised Morse flowline) there is breaking of $w$ at the positive end, concurrent perhaps with breaking of the positive cylindrical end of the domains in $A$.
\end{lemma}
\begin{proof}
Suppose first that the moduli space is of dimension $1$ (i.e. $f=G$). Consider what points need to be added in order to compactify. Given a sequence of points in the moduli space, $[w_n,m_n, u_n]$, if $[w_n(0),m_n]$ is convergent to a boundary point of $G$ such that $m_n$ is away from the boundary of $A$ (which is a point in the moduli space), or $m_n$ converges to a boundary point of $A$, or the omega-limit set. However, by definition the omega limit set is covered by maps $h_i : X_i \rightarrow S^{2i+1} \times_{S^1} A$, where $\dim h_i \le \dim G - 2$. In particular, this implies that any additional compactifying elements can be mapped to one of $\mathcal{M}_{h_i}(y,j)$, but these manifolds are of dimension $\dim(h_i) -|y| - 2\lambda \le \dim(G) -|y| - 2j - 2 \le -1$, hence are empty. 

Further, $m_n$ may not converge to a boundary point of $A$. To see this, suppose first for a contradiction that this were so: in that case, this would imply that exactly one $z_i$ converges to the boundary of $T$ (not two, because this occurs in codimension $2$). If $\eta(i) \neq 0$ then, as with proving that $s_{eq}^1$ is well-defined, we can deduce that such a limiting curve has Chern number $> \lambda$. If $\eta(i) = 0$, then working in the other direction (i.e. the intersection with $M$ is the intersection with $M_{\dagger_\kappa}$ for the $\kappa$ as in our assumptions above) and deduce similarly that the curve $T \subset R$ has Chern number $> \alpha_R(T)$. In either case, there was a contradiction.

The next case is when $[w_n(0),m_n]$ is a convergent subsequence in the open stratum of $G$. Then it is standard $S^1$-equivariant gluing and compactness arguments that tell us we need to a broken configuration as in the statement of the lemma (this carries over as in the case of gluing parametrised holomorphic cylinders, see \cite{bourgeoisoanceafredholm} for the proof of generic regularity). Breaking of $w_n$ at the negative end cannot occur because this would be a codimension $2$ phenomenon (on account of the critical $S^1$-levels of our Morse-Bott function on $ES^1$ existing only in even degrees).

Finally, we need to rule out bubbling. Observe that, in the case as given, the fact that each of $F$, $M$ and $E$ is nonnegative monotone (i.e. there are no holomorphic spheres of negative Chern number) means that standard results as in for example \cite{jholssympl} rules out bubbling (i.e. bubbling occurs in codimension $2$).

If the moduli space is of dimension $0$ (i.e. $f=g_1,g_2$), then as all of the additional cases occur in codimension $1$, for generic choices of data there are no such extra marked points, and hence the moduli space is already compact.
\end{proof}

Suppose first that $A$ is a manifold without boundary, and that the assumptions above hold with $$g_1 = \mu^A_i : D_{2i} \times A \rightarrow S^{2i+1} \times_{S^1} A.$$ Then there is a natural object $$\Psi(\mu^A_i) \in SC^*_{S^1}(E),$$ defined such that the coefficient of $y u^i$ (when the index $|y| -\dim(2 \mu^A_i) +2i= 0$) in $\Psi(\mu^A_i)$ is the number of triples $(w,a,u)$, which are the points in the $0$-dimensional moduli space $\mathcal{M}_{\mu^A_i}(y,j)$ such that $a \in A$, $w \in \mathcal{P}(i,0)$, and $u: m \rightarrow M$ at the point $z \in m$ satisfies the conditions. In the case where $A$ has boundary $N$ with the given assumption, this is replaced by $\mu^A_i \# \mu^{N/S^1}_{i+1}$. Finally, by considering the possible breakings of a $1$-dimensional version of the defining moduli spaces $\mathcal{M}_{\mu^A_i}(y,j)$, one see that $\sum_i \Psi(\mu^A_i)$ is closed, hence $\sum_i \Psi(\mu^A_i) \in SH^*_{S^1}(E)$. 

In particular, possible endpoints of a $1$-dimensional version of the moduli space $\mathcal{M}_{\mu^A_i}(y,j)$ occur when $w$ breaks at its negative end (i.e. $w(0)$ converges to $\partial D_{2i}$), or $w$ and the positive end break in concert. The former can be shown to occur in codimension $2$, hence generically does not occur in $1$-dimensional moduli spaces, and so we obtain the equation: $$ \sum_{k=0}^i d^{i-k} \Psi(\mu^A_{k}) = 0.$$ In particular, when summing over all of the $\Psi(\mu^A_k)$, the resulting object is closed under $d^{eq} = \sum_i d^i$.

Using LbP, recall that given the pseudocycle $\mu_i^A$ (replace $\mu_i^A$ with $\mu^A_i \# \mu^{N/S^1}_{i+1}$ throughout if $A$ has a boundary) associated to $A$, there are codimension $2$ submanifolds $F_k$ and $p^{-1}B_k$ such that $$[\mu_i^A] = \sum_{k : \text{codim}(F_k) = 2} \epsilon(F) [\mu_{i+1}^F] + \sum_k n_k[\mu_{i+1}^{p^{-1}B_k}]$$  attained using some pseudocycle bordism $\nu_i : D_{2i} \times X \rightarrow S^{\infty} \times_{S^1} A$.  We will denote by $C = \bigsqcup_{k : \text{codim}(F_k) = 2} \epsilon(F_k) F_k \bigsqcup_k n_k p^{-1}B$ the union of manifolds, and $\mu^C_i$ to be the pseudocycle $$\sum_{k : \text{codim}(F_k) = 2} \epsilon(F) [\mu_{i+1}^F] + \sum_k n_k[\mu_{i+1}^{p^{-1}B_k} : C \rightarrow A.$$Then suppose that the above assumptions on moduli spaces hold for $$g_2 = \mu_{i+1}^C = \sum_{k : \text{codim}(F_k) = 2} \epsilon(F_k) \mu_{i+1}^{F_k} + \sqcup_k \mu_{i+1}^{p^{-1}B_k}.$$  We can define $\mathcal{M}_{\mu^{C}_i}(y,j)$ consisting of triples $(m,w,u)$ such that $m \in C$, and thus by considering $0$-dimensional moduli spaces as in the previous paragraph (where we have to take the union of the various summands of $\mu_i^C$), we can define an element $\Psi(\mu^C_i) \in SC^*_{S^1}(E)$. As with the previous case, after summing these we obtain $\sum_i \Psi(\mu^C_i) \in SH^*_{S^1}(E)$.
 
Finally, for completeness we observe that for the cycle $\sum_i \Psi(\mu^A_i)$, defined on the chain-level, we see that the chain $u \Psi(\mu^A_i)$ has the coefficient of $y u^k$ in $u \Psi(\mu^A_i)$ is obtained by counting the number of points in $\mathcal{M}_{\mu^A_{i-1}}(y,k-1)$.

\begin{theorem}
\label{theorem:lift-to-moduli}
If the assumptions at the beginning of Section \ref{subsec:lifting-LbP} hold for $G=\nu_i$, $g_1 = \mu^A_i$ and $g_2 = \mu^C_i$ then:

$$\sum_i \Psi(\mu^C_i) = \sum_i u \Psi(\mu^A_i) \in SH^*_{S^1}(E).$$
\end{theorem}
\begin{proof}
One considers the $1$-dimensional moduli space $\mathcal{M}_{\nu}(y,j)$ (necessarily the dimension will be $1$, assuming that $\mathcal{M}_{\mu^C_i}(y,j)$ and $\mathcal{M}_{\mu^A_{i-1}}(y,j)$ are $0$ dimensional) consisting of triples $(w,a,u)$ such that $(w(0),a) \in \text{im}(\nu_i)$ (among the other conditions).

When we compactify this $1$-dimensional moduli space, we now consider what are the boundaries. By construction, there are boundaries associated to the boundary of $\text{im}(\nu_i)$, which corresponds (taking the codimension $1$ limit points of the $m$ coordinate) to $\mathcal{M}_{\mu^C_i}(y,j)$, which we see by LbP, or corresponds (taking the codimension $1$ limit points of $w(0)$, i.e. when $w(0)$ converges to an element of the boundary $\partial(D_{2i}) = S^1 D_{2i-2}$) to a single Morse breaking of $w$ at the negative end of the Morse flowline. In this case, $w$ breaks as $w' \# w_p$ (where by this we mean $w_p(-\infty) = w'(\infty)$), where $w_p$ is the principal component in $\mathcal{F}(k-1,0)$, whereby in this boundary triples $(m,w_p,u)$ satisfy the conditions of triples of $\mathcal{M}_{\mu^A_{i-1}}(y,k-1)$, and $w'$ is in $\mathcal{F}(k,k-1)$. By the same considerations as in LbP, applying the $S^1$-action this exactly bijects with a pair consisting of a point in the moduli space $\mathcal{M}_{\mu_{i-1}^A}(y,k-1)$. Note that the other potential breakings of $w$ on its negative end are impossible because such broken configurations cannot exist for index reasons. 

Further, there are boundaries consisting of fixing some domain $m$ and some $w(0)$ contained in the interior of $D_{2i}$, and then the positive cylindrical end of $m$ breaks at the same time as the positive end of $w$. In this case, $w$ breaks as $w_p \# w'$ where $w' \in \mathcal{F}(l,0)$ and $w_p \in \mathcal{F}(k,l)$, and the cylindrical end breaks. Elements of this boundary thus correspond to a pair consisting of an element of $\mathcal{M}_{\nu_i}(y,l)$ and an element in the moduli space defining $d^{k-l}$ (using $\sigma$-invariance of the data). 

Putting all this together, we obtain a boundary between the chains $\sum_i \Psi(\mu^C_i)$ and $\sum_i u \Psi(\mu^A_i)$.
\end{proof}

\begin{remark}
    
If our space of domains $A$ consisted for example of holomorphic curves with an additional negative cylindrical end, then we could do everything above almost identically. The only difference would be taking into account possible breaking of the negative cylindrical ends of the domains at the same time as the Morse flowlines $w$. The objects $\Psi(\mu^C_i)$ and $\Psi(\mu^A_{i-1})$ above would in that instance be operations $SH^*_{S^1}(E) \rightarrow SH^*_{S^1}(E)$ (as opposed to just elements of $SH^*_{S^1}(E)$), and the proof of Theorem \ref{theorem:lift-to-moduli} modified as we have noted will demonstrate a chain homotopy between the chain level descriptions of these operations $\sum_i \Psi(\mu^C_i)$ and $\sum_i \Psi(\mu^A_{i-1})$, as opposed to finding a boundary between two cochains. 

\end{remark}

\begin{remark}
We note that in the proof of Lemma \ref{theorem:lift-to-moduli}, we no longer take $S^1$-equivalence classes because requiring $w(0) \in D_{2i}$ we break the symmetry.
\end{remark}

\begin{remark}
It should be clear that none of the results of this section depend on the actual choice of symplectic manifold $E$. The narrow set of choices that we made merely allow us to control technical issues such as bubbling, regularity, compactification etc. In particular, there should be extensions of these results to other, more general moduli space constructions. One thing to note is that the lifting procedure might be more complicated in general, relying on the normal bundle of $F$ in $A$ in a nontrivial way. 
\end{remark}

\section{Application to Borman-Sheridan classes}
\label{sec:ourexample}

\subsection{The (equivariant) Borman-Sheridan class, and the equivariant PSS-map}
\label{subsec:operations}

\subsubsection{The Borman-Sheridan class}

To begin, we define the thimble surface $T$ to be $(\mathbb{R} \times S^1) \cup \{ \infty \}$, which is biholomorphic via $z \mapsto 1/z$ to $\mathbb{C}$ (albeit parametrised by a cylinder away from $\infty$). In particular, the action of $S^1$ on $T$ is \begin{equation} \label{equation:rotation} S^1 \times T \rightarrow T, \quad (e^{2 \pi i \theta}, t) \mapsto e^{-2 \pi i \theta} t,\end{equation} for $\theta \in [0,1]$. See Figure \ref{fig:thimble}(A). We choose on this thimble a family of Hamiltonians and almost complex structures $(H^{S^1, \text{thimble}},J^{S^1, \text{thimble}})$ depending on $(v,z) \in ES^1 \times T$ such that each $H^{S^1,\text{thimble}}_{v,z}: E \rightarrow \bR$ is a Hamiltonian and each $J^{S^1,\text{thimble}}_{v,z}$ is an almost complex structure on $F$. We require that when $z=(s,t) \in \bR \times S^1 =  T \setminus \{ \infty \}$, $$\begin{array}{l} H^{S^1, \text{thimble}}_{v,s,t} = H^{S^1, \text{thimble}}_{\theta \cdot v, s, t - \theta} : M \rightarrow \mathbb{R}. \\ J^{S^1, \text{thimble}}_{v,s,t} = J^{S^1, \text{thimble}}_{\theta \cdot v, s, t - \theta} \end{array}$$ for any $\theta \in S^1$, and so that Properties \ref{properties:H} and \ref{properties:J} hold for all $s,v$. We ask that $(H^{S^1, \text{thimble}},J^{S^1, \text{thimble}}) = (H^{eq},J^{eq})$ is independent of $s$ for $s \le 1$, and independent of $v,s,t$ for $s \ge 2$.


We will define $s_{eq} \in SH^{0}_{S^1}(E)$ as follows: for each $k \in \mathbb{Z}_{\ge 0}$ and $\gamma \in \mathcal{P}(H)$, let $$\tilde{\mathcal{M}}(k, \gamma,H^{S^1, \text{thimble}},J^{S^1, \text{thimble}})$$ consist of triples $(\tau,v,u)$ such that \begin{equation}\label{eq:seqprops} \begin{cases}\begin{array}{l} \tau \in T, \\ w \in \mathcal{F}(k,0) \\ w(-\infty) \in c_k, \ w(\infty) \in c_0 \\ u : T \rightarrow F \\ \partial_s u + J^{\text{thimble}}_{w(s),s,t}(u)(\partial_t u - X^{\text{thimble}}_{w(s),s,t}) = 0 \\ u(\zeta) \in M \\ \lim_{s \rightarrow -\infty} u(s,t) = \gamma(t), \\ u \cdot [M] = 1, \\ u(\tau) \in M_{\dagger}. \end{array}\end{cases}.\end{equation} Further, this comes equipped with a free, properly discontinuous $S^1$-action, and we define $$\mathcal{M}(k,\gamma,H^{S^1,\text{thimble}},J^{S^1,\text{thimble}}) = \tilde{\mathcal{M}}(k, \gamma,H^{S^1,\text{thimble}},J^{S^1,\text{thimble}})/S^1.$$ This is a manifold of dimension $|\gamma| + 2k$. 

We must now demonstrate that, when the moduli space is of dimension $0$, it is compact. Further, if the moduli space is of dimension $1$ then it can be compactified by adding setups consisting of concurrent breaking of the negative cylindrical end of $T$, along with breaking on the negative cylindrical end of $w$. Ideally, we would like to demonstrate generically any compactifying curves must sit in codimension at least $2$, as is standard. We proceed as in \cite{lefschetz3}. There are three important types of compactification that we have to consider: 
\begin{itemize}
    \item when $\tau = (s,t)$ in the cylindrical end of $T$, and $s \rightarrow \infty$. Notice however that such a broken setup consists of an equivariant (perturbed) holomorphic thimble attached to an equivariant Floer trajectory at their cylindrical ends, such that the equivariant trajectory intersects $M_{\dagger}$. Such solutions cannot be excluded for index reasons, because they are of codimension $1$. So in order to make $s_{eq}$ well-defined, one needs to rule out such a situation. 
    
    However, to do this we observe that all of the curves in the open part of the moduli space have $u \cdot [M] = 1$, hence this holds for broken configurations. But we notice that, in the broken configuration, the equivariant Floer trajectory has an intersection with $M_{\dagger}$. By the choice of $\dagger$ (in a trivialising neighbourhood for $F \rightarrow \bC P^1$ of $0$), this means that this trajectory will also generically have an intersection with $M$. But now the broken trajectory has at least $2$ intersections with $M$, which is a contradiction. This eliminates such boundary points from occuring in the compactification.

\item breaking on the positive end of $w$ cannot occur for moduli spaces of dimension $0,1$, as this must occur in codimension $2$ (as all critical values are in even degree).
    
    \item we now need to demonstrate that bubbling cannot occur. Because $F$ is nonnegatively monotone, we can use usual methods to rule out bubbling everywhere except when $\tau = \infty \in T$. Note in particular that when $\tau =  \infty$, the resulting stable sphere that bubbles off from $T$ from the collision of $\tau$ with the marked point at $\infty$ cannot be contained in $M$, because this would imply that $M_{\dagger}$ and $M$ intersect. 
    
    The only other case is when a bubble occurs entirely within $M$ at $\infty \in T$ (or an unstable bubble at the equivalent point in the nodal configuration when $\tau = \infty$). However, because we have assumed that $M$ is nonnegatively monotone, this ensures that such situations occur in codimension $2$ (in particular, all such bubbled spheres must have non-negative Chern number). 
\end{itemize}

Now that we know this, we can define $$s_{eq} = \sum_{\begin{array}{c} (c_k,\gamma) \in SC^*_{eq}(E), k \in \mathbb{Z}_{\ge 0}, \\ |\gamma| + 2k = 0 \end{array}} \# \mathcal{M}(k, \gamma,H^{S^1,\text{thimble}},J^{S^1,\text{thimble}}) \cdot \gamma u^k.$$ One can also demonstrate that this is closed under $d^{eq}$, by considering a $1$-dimensional version of the moduli space $\mathcal{M}(k,\gamma,H^{S^1,\text{thimble}},J^{S^1,\text{thimble}})$. The degree arguments and bubbling arguments still prevent respectively limiting curves when $\tau$ travels along the cylindrical end of $T$, and sphere bubbling. However, there can be limiting curves consisting of a breaking of $T$ along the cylindrical end, along with a concurrent breaking of $w$ at the negative end (breaking of $w$ at the positive is discounted because the resulting limited configuration would be a product of two moduli spaces, one of negative virtual dimension). This in turn provides us with the following boundary for the $1$-dimensional moduli space $\mathcal{M}(l,\delta,H^{S^1,\text{thimble}},J^{S^1,\text{thimble}})$ (recalling the notation \eqref{equation:eq-diff} for the moduli space defining the equivariant differential):
$$\bigsqcup_{j,\gamma} \mathcal{M}(k,\gamma,H^{S^1,\text{thimble}},J^{S^1,\text{thimble}}) \times \mathcal{M}(k,\gamma,l,\delta).$$

In turn, by counting the points in this boundary one obtains that $d^{eq} ( s_{eq}) = 0$. Further, from standard equivariant continuation arguments we can similarly see that this is independent of our choices.





\subsubsection{The equivariant PSS-map}

We define the $S^1$-equivariant PSS-map. The version we use will equate cohomology with locally finite homology via Poincar\'e duality, and thus the notion of pseudocycles is as in \cite{lf-pseudocycles}. 

We pick some generic pseudocycle $f_y: Y \rightarrow E$, representing $y \in H^*(E)$. Given also a Hamiltonian loop $x$, define $\tilde{\mathcal{M}}^{\text{PSS}}(j, y, k, \gamma,H^{S^1,\text{thimble}},J^{S^1,\text{thimble}})$, which consists of triples $(\tau,v,u)$ satisfying:

\begin{equation}\label{eq:seqprops3} \begin{cases}\begin{array}{l} w \in \mathcal{F}(k,j), \\ u : T \rightarrow E \\ \partial_s u + J^{\text{thimble}}_{w(s),s,t}(u)*(\partial_t u - X^{\text{thimble}}_{w(s),s,t}) = 0 \\ u(\zeta) \in f_y(Y), \\ \lim_{s \rightarrow -\infty} u(s,t) = \gamma(t), \\ \text{deg}(u) = 1. \end{array}\end{cases}.\end{equation}

Then the coefficient of $\gamma u^k$ in $\text{PSS}_{eq}(y u^j)$ is $\# \tilde{\mathcal{M}}^{\text{PSS}}(j, y, k,\gamma,H^{S^1,\text{thimble}},J^{S^1,\text{thimble}}) / S^1$.

All of the methods from the previous section (breaking, bubbling etc) demonstrate that this yields a well-defined map $$\text{PSS}_{eq}: H^*(E) \otimes H^*(BS^1) \rightarrow SH^*_{S^1}(E).$$

\begin{remark}
One should compare this with the definition of the PSS map using Morse cohomology, by replacing the intersection of $u(\infty)$ and $f_y$ with intersection of $u(\infty)$ and the stable manifold of some critical point of the Morse function. Constructing things that way, one would be forming the locally finite Morse pseudocycle by ``gluing together" stable manifolds of critical points (as in \cite{schwarzmorsesingiso}) and recover the definition above.
\end{remark}
 


\subsection{Statement 1}
\label{subsec:statement1}

We would like to apply the result of Section \ref{sec:lifting-to-moduli-spaces} in a simple example. 

Before the theorem, we denote the following pseudocycle: $\mathcal{M}$ consists of degree $1$, $J$-holomorphic maps $u: S^2 \rightarrow F$ with $3$ distinct marked points, $z_1=\infty,z_2=1,z_3=0$ such that $u(z_1) \in M$ and $u(z_2) \in M_{\dagger}$ and $z_3=0$ is unconstrained. Then there is an evaluation map $$z^{(1)}: \mathcal{M} \rightarrow F, \quad z^{(1)}(u) = u(z_3),$$ and our condition $u \cdot M = 1$ implies that $u(z_3) \notin M$ (as $u(z_1) \in M$) and so the evaluation map defines a pseudocycle $z^{(1)}|_E : \mathcal{M} \rightarrow E$. Standard Gromov compactification arguments imply that this is pre-compact and a pseudocycle.

We will prove the following theorem, a conjecture \cite[Conjecture 3.3]{lefschetz3} by Seidel:

\begin{theorem}
\label{theorem:ourexample}
In $SH^2_{eq}(E)$, 

\begin{equation}
\label{equation:result1}
    u s_{eq} = PSS_{eq}(z^{(1)}|_E),
\end{equation}


\end{theorem}

\begin{figure}
    \centering
    \includegraphics[scale=0.7]{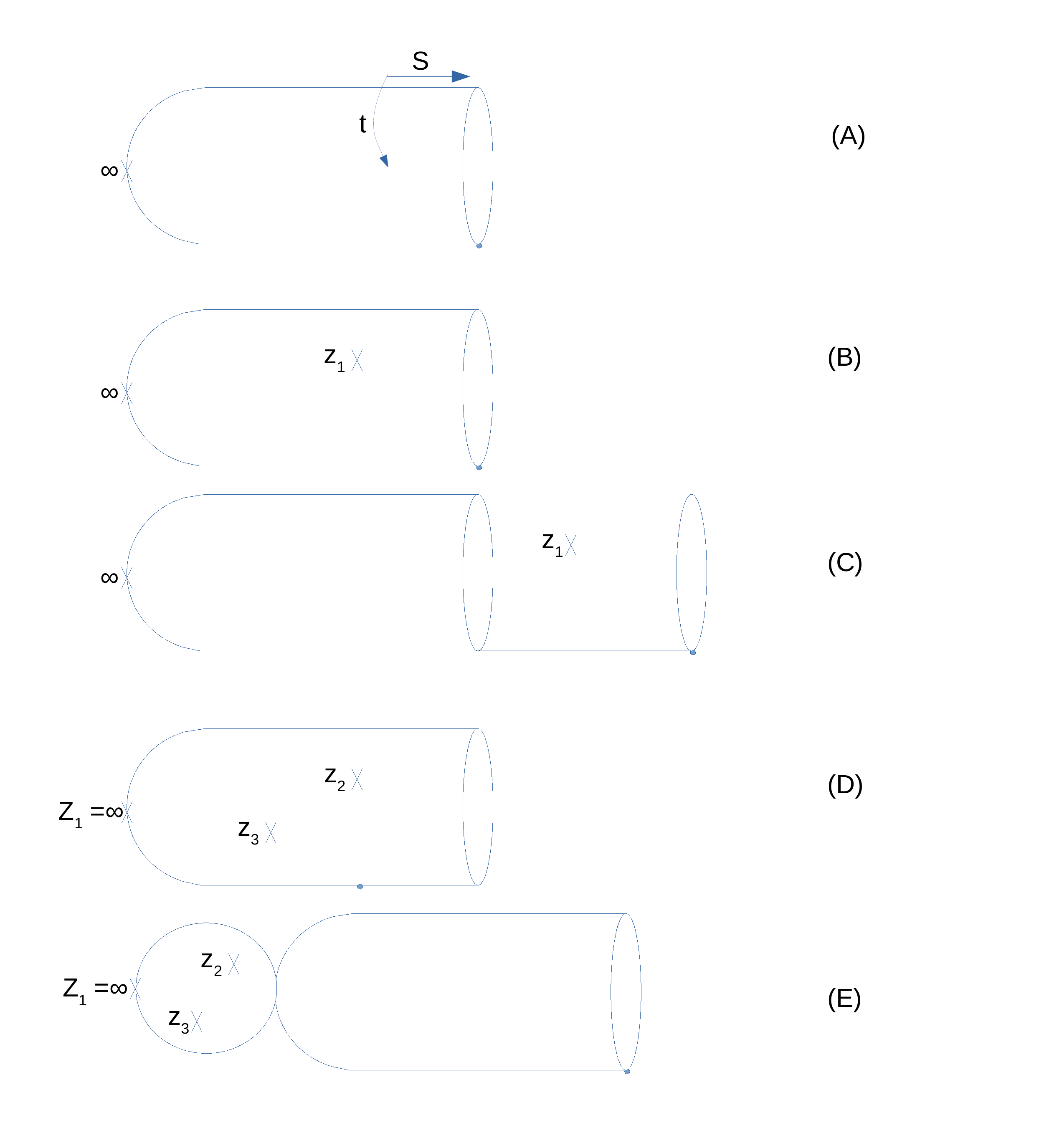}
    \caption{(A) is the holomorphic thimble. (B) is the space of domains as used in Section \ref{subsec:statement1}. (C) is the broken domains when the marked point reaches the boundary. (D) is the space of domains used in Section \ref{subsec:statement2}. (E) is the exceptional divisor of $Bl_{(\infty,infty)}(T^2)$.}
    \label{fig:thimble}
\end{figure}

Before proceeding to the proof of \eqref{equation:result1}, we will consider how this relates to the setup of Section \ref{subsec:lifting-LbP}. In particular, in this case our manifold of domains $A$ will consist of the holomorphic thimble $T$ with $l=0$ bubbles attached, with $m=2$, $\lambda=1$, and $\eta(z_1) = 0$, $\eta(z_2) = \dagger$. In particular, the parameter space consists of a single freely moving marked point $z_2 \in T$, additionally to a single marked point with position fixed at $z_1 = \infty \in T$. See Figure \ref{fig:thimble}(B). Then $A \cong T$ (a fixed interior point and an asymptotic point of $T$ fix the parametrisation). One can compactify $A$ to a closed $2$-disc by adding in the boundary situation. In this case, such a setup would correspond to the marked points $z_2$ travelling to the cylindrical end, and therefore we would naturally obtain a Floer breaking of the cylindrical end of the thimble. The $t$-coordinate of $z_2$ thus would determine the limiting setup of such an element of the compactification, and therefore one obtains a closed $2$-disc as the compactified manifold-with-boundary. Note however that this choice of compactification will be broadly irrelevant, as the same argument we used to show that the cycle $s_{eq} \in SH^*_{eq}(E)$ is well-defined, will also determine that the projection of our parametrised $1$-dimensional moduli spaces to $\overline{A}$ may never reach the boundary. See Figure \ref{fig:thimble}(C) for a point in this compactification. Finally, we set our Chern signature to be $\lambda=1=\alpha(T)$.

We now need to demonstrate that $A$ and $F$ satisfy the assumptions at the beginning of Section \ref{subsec:lifting-LbP}: in particular, that we can define smooth moduli spaces of the correct dimension over them. We do not need to check the moduli spaces for the pseudocycle bordism between, because if necessary we can always perturb this bordism relative to its boundary. Standard bubbling and compactification arguments prevent any bubbling from occurring when defining $\Psi(\mu^A_i)$. For $\Psi(\mu^F_i)$, note that the nonnegative monotone condition and the degree $1$ condition immediately prevent all but the bubbling of multiply covered spheres, and in fact any multiply covered spheres must be entirely in a fibre (i.e. they are of degree $0$). Any simple sphere bubbling entirely contained in the fibre only occurs in codimension $2$, because we chose our fibres to be nonnegatively monotone (by replacing a multiply covered sphere bubble by the underlying simple sphere bubble).



\begin{proof}[Proof of \eqref{equation:result1}]
The space of domains in our setting is (after compactification) isomorphic to $\overline{D^2}$, the closed $2$-disc, corresponding to the thimble with fixed parametrisation, with a single marked point. In keeping with our previous definitions, we will parametrise $\overline{D^2}$ such that it is $\bC \cup \{ \infty \}$ with the unit disc around $0$ removed. The $S^1$-action is rotation. 


It remains to assess what LbP and Theorem \ref{theorem:lift-to-moduli} tells us about the resulting operations.  The fixed-point-set of $D^2 \cong \bC \cup \{ \infty \} \setminus \{ 0 \}$ with $S^1$ being the rotation \eqref{equation:rotation} is $\infty$, which is of codimension $2$. Further, $\overline{D^2} \setminus \{ \infty \} \rightarrow [1,\infty)$ is a trivial $S^1$-bundle. Therefore, we can apply the first case of LbP in the simplest situation (where $B = \emptyset$, and by obvious convention $[\mu_i^{p^{-1}\emptyset}] = 0$). In this case, $$[\mu^{\overline{D^2}}_{i-1} \# \mu^{\partial \overline{D^2}/S^1}_{i}] = [\mu^{\infty}_{i}]:$$ to obtain the sign $\epsilon(\{ \infty \}) = 1$, notice that our choice of orientation for $\widecheck{W}$ requires that the image of the section of $\text{Bl}_{\infty}T \rightarrow (0,\infty]$ be a radial line oriented facing towards $\infty$. In particular, the boundary at $\infty$ in $\text{Bl}_{\infty}T / S^1$ is a positively oriented point in this $1$-dimensional moduli space. Looking at the respective moduli spaces as constructed in Theorem \ref{theorem:lift-to-moduli}, we observe exactly that (because no element of the resulting moduli space may sit above some neighbourhood of $\partial \overline{D^2}$, due to the arguments used to define $s_{eq}$ in Section \ref{subsec:operations}) the operation as in Theorem \ref{theorem:lift-to-moduli} defined by $[\mu^{\overline{D^2}}_{i-1} \# \mu^{\partial \overline{D^2}/S^1}_{i-1}]$ is the coefficient of $y u^k$ in $u s_{eq}$. 

It remains to describe $[\mu^{\infty}_{i}]$. The moduli space we construct, denoted $\mathcal{M}_{\mu^{\infty}_i}(y,1)$ in the language of Section \ref{sec:lifting-to-moduli-spaces}, uses as its space of domains $F$, which consists of a perturbed holomorphic sphere attached to a holomorphic thimble at a single point. The perturbed image of the holomorphic sphere lies in $F$, passes through $M \subset F$, and passes through the submanifold $M_{\dagger}$. The holomorphic thimble meets this perturbed holomorphic sphere at $\infty$, and it is asymptotic at its cylindrical end to $y$. By choosing a homotopy of the induced almost complex structure and Hamiltonian perturbation on the bubbled sphere, we may assume (up to an exact element of $SC^*_{eq}(E)$) that we have no Hamiltonian perturbation on the bubbled sphere, and the almost complex structure is the structure on $F$ satisfying Properties \ref{properties:J}. After this homotopy of the data, we see that the evaluation map at the nodal point of the sphere is the pseudocycle $z^{(1)}|_E$. Finally, the number of maps on the holomorphic thimble with the given intersection condition at $\infty$ is the coefficient of $y u^k$ in $PSS_{eq}(z^{(1)}|_E)$.
\end{proof}


\begin{remark}
Here it is also worth pointing the reader towards the intertwining relation proved by Todd Liebenschutz-Jones in \cite{todd}. One can consider this as a very similar result to that of he cited paper: indeed, if one were to apply the lifting method of LbP to the situation in \cite{todd}, one would obtain exactly the same proof of the intertwining relation. The only real difference between this result and that of \cite{todd}, is the difference between applying LbP to $\overline{D^2}$ versus $S^2$.
\end{remark}

\subsection{Statement 2}
\label{subsec:statement2}
We now demonstrate how to pass to the next level, i.e. consider degree $2$ curves. 

In this instance, we will apply lifted LbP in multiple iterations, noting that we can only do so because our fixed-point-set is of codimension $2$ (see Remark \ref{remark:iteration}). The first space of domains $A$ will consist of: the Riemann surface $T$ with four interior marked points (i.e. $m=4$). There is an obvious $S^1$-action induced by the diagonal, and up to reparametrisation (one of the marked points, say the first, may be fixed at $\infty \in T$) this manifold-with-corners is of dimension $6$, where the open stratum is $(\bC \cup \{ \infty \} \setminus \{ 0 \})^3$, with the compactification being obtained by adding the exceptional divisors of iterated blowups of $T^3$. To compactify, we allow any number of the marked points may reach the boundary (this will correspond in the space of holomorphic maps to breaking of Floer trajectories on the cylindrical end, with the resulting marked point or marked points being at the $s=0$-circle of the resulting cylinder). The intersection conditions we make at the marked points are $M,M,M_{\dagger},M_{\dagger}$, where $\dagger$ is as defined earlier: some point in a neighbourhood $0 \in B(\epsilon,0)$ above which $p^{-1}B(\epsilon,0) \cong M \times B(\epsilon,0)$. We will use coordinates $(z_1,z_2,z_3,z_4)$ for the four marked points in the open stratum, but recall that on the open stratum $z_1 = \infty$ is fixed. Then $S^1$ acts by diagonal multiplication by $e^{-2 \pi i \theta}$. See Figure \ref{fig:thimble2}(A).

We should reiterate that this is not a smooth manifold with boundary: there are corners. All of the previous work could be extended for locally finite pseudocycles as in \cite{lf-pseudocycles}, but this is presently unnecessary using the following trick: within $A$ there is a sequence of smooth manifolds with boundaries, $A_1 \subset A_2,\dots$, fixed by the $S^1$-action, which are obtained by taking $A$ and removing a small (i.e. the volume of $A \setminus A_n$ converges to $0$ as $n \rightarrow \infty$) neighbourhood of the corners. As with the proof that our definition of $s_{eq}$ was well-defined, one can demonstrate that our moduli space of consideration, which counts degree $2$ holomorphic maps, is parametrised by $A_N$ for some $N$ sufficiently large: else, one would be able to find a limiting holomorphic map of degree $>2$. Hence without loss of generality we replace $A = A_N$ for some sufficiently large $N$.

\begin{figure}
    \centering
    \includegraphics[scale=0.5]{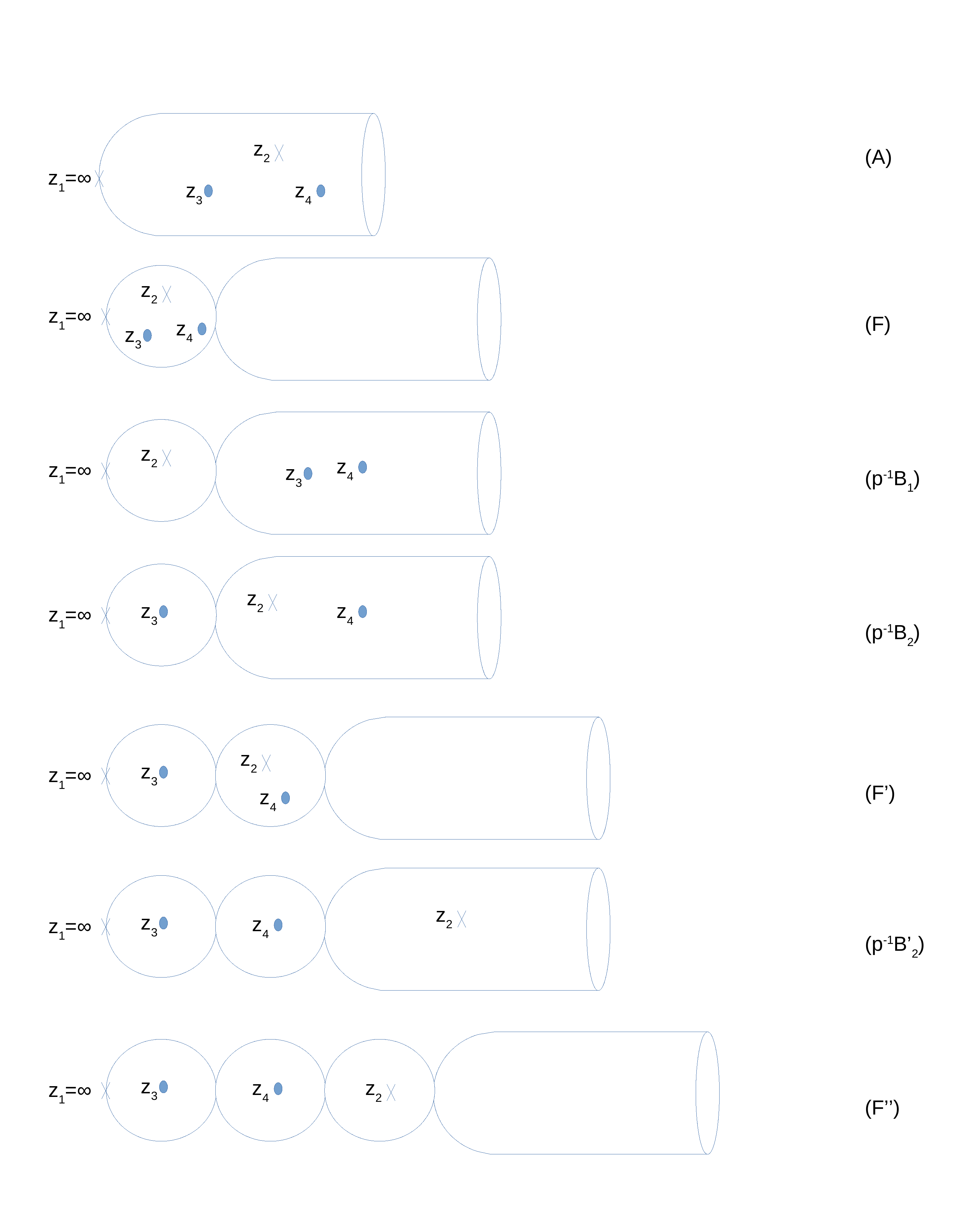}
    \caption{Various configurations of points in $A$. Alternatively a $\times$ has been used to denote a marked point the image of which will intersect $M$, and a filled circle denotes a marked point the image of which will intersect $M_{\dagger}$ (or $M_{\dagger_i}$ after homotopy).}
    \label{fig:thimble2}
\end{figure}

We observe that, beginning with the $S^1$-manifold $A$, there are two distinct choices for applying LbP. For ease of notation, we will replace $\mu^X_{i-1} := \mu^X_{i-1} \# \mu_i^{\partial X / S^1}$, this being unambiguous on the level of pseudocycles. We can obtain:

\begin{equation}
\label{equation:1st-degree2}
    [\mu^A_{i-1} ] = [\mu^F_i] + [\mu^{p^{-1} B_1}_i]
\end{equation}
\begin{equation}
\label{equation:2nd-degree2}
    [\mu^A_{i-1} ] = [\mu^F_i] + [\mu^{p^{-1} B_2}_i]
\end{equation}
 where $F$ consists of the configuration where all of the marked points have bubbled off together at $\infty \in T$, and $p^{-1}B_1$ consists of the configuration where $z_2$ collides with $z_1$, and $p^{-1} B_2$ consists of the configuration where $z_3$ collides with $z_1$. See Figure \ref{fig:thimble2}($F$),($p^{-1}B_1$),($p^{-1}B_2$). In order to see that this is the case, we first recall that the $S^1$-action (as in \eqref{equation:rotation}) is induced by $$(e^{2 \pi i \theta}, (x,y)) \mapsto (e^{-2 \pi i \theta} x, e^{-2 \pi i \theta} y)$$ for $\theta \in [0,1]$, which satisfies all of the Assumptions \ref{assumption:ass1}, \ref{assumption:mfdwbdry-assumption} (the latter of which we will see later). 
 
 As noted above, the fixed point set $F$ is obtained in the case where all of the marked points reach $\infty$ (and thus bubble off). We now consider $A \setminus F$. Removing $p^{-1}B_1$ is sufficient to trivialise the associated $S^1$-bundle, because outside of $p^{-1}B_1$, the point $z_2$ must sit on the principle component, $T$, of the bubble configuration. Therefore, on the complement of $p^{-1}B_1$, the value $\text{arg}(z_2) \in S^1$ is well defined. Hence, there is a natural section of the $S^1$-bundle $$T \setminus \{ \infty \} \times  T \times T \rightarrow (T \setminus \{ \infty \} \times  T \times T)/S^1,$$ obtained by $[z_2,z_3,z_4] \mapsto (|z_2|,z_3e^{- i \text{arg}(z_2)},z_4e^{- i \text{arg}(z_2)})$, which extends to a section of $$A \setminus F \setminus p^{-1}B_1 \rightarrow (A \setminus F \setminus p^{-1}B_1)/S^1.$$ The same argument holds for $B_2$, swapping the roles of $z_2$ and $z_3$. We demonstrate in Appendix \ref{sec:orientations} that the signs are correct.
 
  Equations \eqref{equation:1st-degree2} and \eqref{equation:2nd-degree2} together inform us that $[\mu^{p^{-1} B_1}_i] = [\mu^{p^{-1} B_2}_i]$. Next we restrict attention to $p^{-1}B_2$. We recall that this corresponds to all points in $A$ such that $z_1$ and $z_3$ have bubbled off together at $\infty$ (and $z_2,z_4$ are freely moving). We can now iteratively apply LbP to $[\mu^{p^{-1} B_2}_i]$, in the same way obtaining the following relation:

\begin{equation}
\label{equation:3rd-degree2}
    [\mu^{p^{-1} B_2}_i] = [\mu^{F'}_{i+1}] + [\mu^{p^{-1} B'_2}_{i+1}].
\end{equation}

In this case where the fixed-point set $F'$ consists of a sphere $L$ containing $z_1=\infty,z_3=1$ and a sphere $R$ containing $z_2=1, \ z_4$ freely moving, such that $L,R$ are attached respectively at $0,\infty$, and $R$ is attached to a copy of $T$ respectively at $0,\infty$. We choose $B'_2$ such that $z_4$ converges to $\infty$, thus $p^{-1}B'_2 \subset A$ consists of two spheres in a chain connected to $T$ (i.e. a nodal holomorphic curve as in $F'$) with the marked points $z_1,z_3$ on the first sphere, $z_4$ on the second, and $z_2$ varying freely across $T$. See Figure \ref{fig:thimble2}($F'$),($p^{-1}B'_2$).

Finally, we observe that we may apply LbP as in Section \ref{subsec:statement1} to $p^{-1}B'_2$ to observe that:

\begin{equation}
\label{equation:4th-degree2}
    [\mu^{p^{-1} B'_2}_i] = [\mu^{F''}_{i+1}],
\end{equation}

where the fixed-point set $F''$ consists of spheres $L,C,R$ connected in a chain, with $R$ connected to a copy of $T$, such that $z_1,z_3$ are on $L$, $z_4$ is on $C$ and $z_2$ is on $R$. See Figure \ref{fig:thimble2}($F''$).

Now, by putting together all of \eqref{equation:1st-degree2},\eqref{equation:2nd-degree2},\eqref{equation:3rd-degree2},\eqref{equation:4th-degree2}, we obtain:

\begin{equation}
\label{equation:final-degree2}
    [\mu^{p^{-1} B_1}_i] = [\mu^{F'}_{i+1}] + [\mu^{F''}_{i+2}].
\end{equation}

Now it remains to lift LbP and interpret the resulting operations, i.e. we need to interpret the results of Section \ref{subsec:lifting-LbP} in this context. We have already described $A$, the choices of marked points, and the intersecting pseudocycles. We note that, in order to reasonably interpret our moduli spaces, we will need to make appropriate choices of $H,J$ as given. We pick the degree condition $\lambda=2$. Demonstrating that our given pseudocycles $f$ induce smooth moduli spaces is as in Section \ref{subsec:statement1}.

Interpreting the terms of Equation \eqref{equation:final-degree2}, we first note that when considering the moduli spaces used to define $\Psi ( \mu_i^{p^{-1} B_1})$, we may first choose a homotopy of the auxiliary data. Suppose we simultaneously choose homotopies between:

\begin{itemize}
    \item the induced almost complex structure on the bubbled sphere and the fixed almost complex structure from Setup \ref{properties:J},
    \item the Hamiltonian perturbation on the bubbled sphere, and $0$.
    \end{itemize}
    
    Because we may choose these homotopies relative to the cylindrical end of $T$, by considering the endpoints of the $1$-dimensional moduli space induced by this homotopy of the data, we see that the resulting closed elements of $SC^*_{eq}(E)$ associated to these two choices of data will differ by an exact element of $SC^*_{eq}(E)$. After passing to this new choice of data, the only possible topological situation that could occur when using the pseudocycle $\mu_i^{p^{-1} B_1}$ is when the holomorphic sphere containing $z_1,z_2$ is constant: if this were not so, the resulting map on the nodal curve would necessarily have Chern number $>2$. This is because the degree of the thimble $T$ is at least $1$ (as it intersects e.g. $M_{\dagger}$), hence if the sphere containing $z_1,z_2$ were nonconstant it must have Chern number $\ge 2$, hence then the curve we obtain would have degree $\ge 3$. Bearing this in mind, when defining $\Psi ( \mu_i^{p^{-1} B_1})$ we are counting holomorphic thimbles with a tangency (with a factor of $2$ arising from free choice of the placement of the intersections with $M_{\dagger}$ in $T$). We will define $$\tilde{s}^{(2)}_{eq} := \sum_i \tfrac{1}{2} \Psi(\mu_i^{p^{-1}B_1}) \in SH^*_{S^1}(E).$$

When considering the moduli spaces for $\Psi(\mu^{F'}_{i+1})$, we note that we are determining the coefficient of $u^{i+1}$ in $PSS_{eq}(z^{(1)} *_F^{(1)} [M])$. As previously, we will choose a homotopy of our almost complex structure and Hamiltonian: however, we will concurrently take a homotopy of our intersection condition $u(z_4) \in M_{\dagger}$ to $u(z_4) \in M_{\dagger'}$, such that $\dagger'$ is in a trivialising neighbourhood of $0 \in \bC P^1$, and the line segment between $\dagger$ and $\dagger'$ does not pass through $0$. As before, this homotopy of the defining data induces an exact chain between the cycles on defines on $SC^*_{eq}(E)$. In particular, we may assume on the cohomology level that the operation $$\sum_i \Psi(\mu^{F'}_{i+1}) = u PSS_{eq}(z^{(1)} *_F^{(1)} [M]) \in SH^*_{eq}(M).$$ It is important that all of the homotopies, both for the data and for the intersection conditions, occur away from the cylindrical end of the attached thimble (in fact the condition is stronger, in this case the homotopy is relative to the entirety of $T$ as opposed to just its cylindrical end). 

When considering the moduli spaces for $\Psi(\mu^{F''}_{i+2})$, we likewise choose concurrent homotopies of our data on $L,C,R$ as previously: on the bubble tree, we pick a path between our induced almost complex structure and the chosen almost complex structure as in Properties \ref{properties:J}, and a path between our Hamiltonian perturbation and $0$, and between our intersection condition with $M_{\dagger}$ and $M_{\dagger'}$. We then note that the holomorphic sphere on $C$ is of degree $0$, hence necessarily constant when the moduli space is of dimension $0$. To see this, note that it is important to recall that $F''$ was obtained as a submanifold of $p^{-1}B'_2$. In the moduli space parametrised by $p^{-1}B'_2$, the bubbled sphere containing $z_4$ is necessarily of topological degree $0$ (else, as with previous arguments, this would contradict that the map on the overall nodal sphere has degree $2$). Hence, this is also true for the moduli space over the submanifold $F'' \setminus p^{-1}B'_2$ (i.e. the moduli space parametrised by $F''$ has the induced Chern signature from $p^{-1}B'_2$). What we are determining is therefore the coefficient of $u^{i+2}$ in $(z^{(1)} \cup [M]) *^{(1)}_F [M]$ restricted to $E$, where we are using that the homology class does not depend on the fibre, i.e. $[M_{\dagger'}] = [M]$.

Putting this all together, we obtain \begin{equation} \label{equation:paul-conjecture} 2 u^2 \tilde{s}^{(2)}_{eq} = PSS_{eq}(u z^{(1)} *_F^{(1)} [M] \Bigr|_E +  (z^{(1)} \cup [M]) *^{(1)}_F [M] \Bigr|_E ). \end{equation} This matches \cite[Conjecture 3.3]{lefschetz3}, in the form of \cite[Equation (3.23)]{lefschetz3} after application of \cite[Equation (3.16)]{lefschetz3}.


\begin{remark}
    \label{remark:iteration}

    Using the terminology of Section \ref{sec:LbP}, suppose that the fixed point set of $A$ splits into multiple connected components. For each connected component of the fixed point set $F_i \subset A$ such that $\lambda_i = 2$, we may assume that $B$ intersects transversely with $E_{F_i} / S^1 \cong F_i$. In particular, this means that one may make a choice of $B$ such that in a neighbourhood of $F_i$, the space $BD(p^{-1}B)$ is smooth. Broadly, one does this by making sure that in some tubular neighbourhood of $F_i$, this $B$ is a union of fibres: this is a straightforward smoothing operation for a general $B$ that intersects $F_i$ transversely, which e.g. one can obtain by way of some simple results on vector bundles.

    More generally, however, there is no such guarantee of smoothness for $BD(p^{-1}B)$ at connected components of the fixed point set $F_j \subset A$ such that $\lambda_j > 2$. Indeed, in general $B$ will need to intersect every fibre of such $E_{F_j}/S^1$. To see this, note that restricting the principal $S^1$-bundle to such a fibre will be of the form $S^{2 \lambda_j - 1} \rightarrow S^{2 \lambda_j -1} / S^1$. This itself is never a trivial $S^1$-bundle for $\lambda_j >1$. This means that $B$ must hit every fibre of the unit normal bundle $E_{F_j}$ of $F_j$, in particular $B$ does not intersect $E_{F_j}$ transversely. In some loose sense, one can say that ``$B$ contains $F_j$", although this is a nonrigorous statement. Regardless, there is no longer a guarantee that $BD(p^{-1}B)$ is a smooth closed manifold in such cases, and therefore it is not possible to iterate. 

    Thus we note that the reason why we could iterate in this section (and why we can also do so in Section \ref{subsec:statementn}) is because at each stage the fixed point set is of codimension $2$ in the parameter space. Without this stipulation, iteration becomes more difficult (as written, one would need to directly prove that such $F_j$ embed into $BD(p^{-1}B)$ as smooth submanifolds). Methods are in production to deal with this issue in generality. 
\end{remark}

\subsection{Statement $n$}
\label{subsec:statementn}
In order to approach the result for general $n$, we will need some notation. Throughout this section, we will fix this $n \in \bZ_{\ge 1}$. For $i=0,1,\dots$ we denote by $S(i)$ the nodal holomorphic curve consisting of a chain of $i+1$ spheres labelled $S(i)_0,\dots,S(i)_i$, such that $S(i)_j$ is attached to $S(i)_{j+1}$ respectively at $0,\infty$ for $j=0,\dots,i-1$, and $S(i)_{i}$ is attached to a holomorphic thimble $T$ respectively at $0,\infty$. We will set the notation $S(-1) := T$.

The moduli space will consist of such nodal curves with $2n$ marked points, labelled $x_1,\dots,x_n,y_1,y_n$, and the intersection conditions will be $M$ for each $x_i$. We will pick distinct $\dagger = \dagger_1, \dagger_2, \dots,\dagger_n \in B(\epsilon,0) \subset \bC P^1$ and $M_{\dagger_i} = p^{-1}(\dagger_i)$, such that the line segments between $\dagger$ and $\dagger_i$ for $i\neq 1$ do not intersect each other or $0$. This will be used later: to begin with, we require the intersection condition at each $y_i$ to be $M_{\dagger}$.

Given the nodal curve $S(i)$, for each $(i+2)-$tuple $(A_0,A_1,\dots,A_i,A_{i+1})$ of pairwise disjoint subsets of $\{ x_1,\dots,x_n,y_1,\dots,y_n \}$ such that $\{ x_1,\dots,x_n,y_1,\dots,y_n \} = \bigsqcup_{j=0}^{i+1} A_j$, one can determine a smooth manifold consisting of marked nodal holomorphic spheres as follows:

\begin{itemize}
    \item for each such $(i+2)$-tuple, the underlying nodal holomorphic curve is $S(i)$, 
    \item for the marked points we require that for $j=0,\dots,i$ the $A_j$ denotes the set of marked points that lie on $S(i)_j$. The $A_{i+1}$ denotes marked points on $T$. 
    \item After quotienting by reparametrisation, we will fix some number of the marked points, decide by iteratively choosing the ``largest" using the ordering $$x_1 > x_2 > \dots > x_n > y_1 > \dots > y_n,$$ (note that this ordering can also induce an ordering on each $A_j$, making these into tuples as opposed to sets). 
    \item The smooth manifold will be determined by choice of position of the other marked points. 
    \end{itemize}
    
    We will keep the abusive notation $S(-1)$ to denote the unique $1$-tuple $(\{ x_1,\dots,x_n,y_1,y_n \})$, i.e. the marked Riemann surface consisting of all $2n$ points on $T$. We will also, for ease of notation, denote the smooth manifold $$[A_0,A_1,\dots,A_i] := \left(A_0,A_1,\dots,A_i, \{ x_1,\dots,x_n,y_1,y_n \} \setminus \bigsqcup_j A_j \right),$$ given any pairwise disjoint set of subsets $A_0,A_1,\dots,A_i$ of $\{ x_1,\dots,x_n,y_1,y_n \}.$ One can see examples of such space of nodal curves in Figure \ref{fig:thimble3}, as taken from the proceeding arguments. Further, these smooth manifolds can be compactified by adding the usual Deligne-Mumford sphere bubbling when marked points collide with each other and nodes.

\begin{figure}
    \centering
    \includegraphics[scale=0.45]{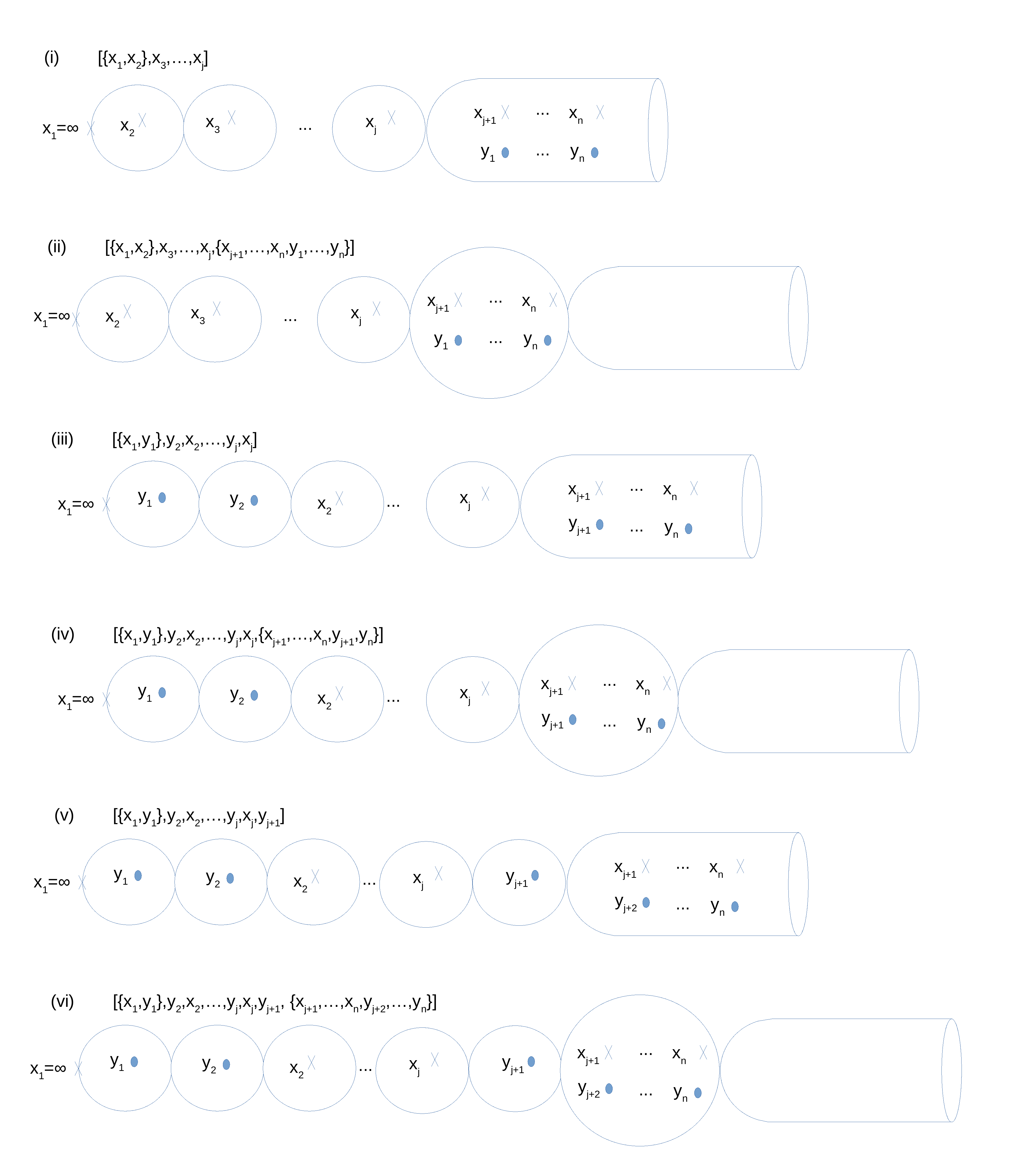}
    \caption{Various choices of spaces of nodal curves, with the given notation, from (top to bottom): (i) the left (and right) hand side of \eqref{equation:3rd-degreen}; (ii) the right hand side of \eqref{equation:3rd-degreen} and \eqref{equation:4th-degreen}; (iii) the left hand side of \eqref{equation:5th-degreen} and the right hand side of \eqref{equation:6th-degreen}; (iv) the right hand side of \eqref{equation:5th-degreen}; (v) the left hand side of \eqref{equation:6th-degreen} and the right hand side of \eqref{equation:5th-degreen}; (vi) the right hand side of \eqref{equation:6th-degreen}.}
    \label{fig:thimble3}
\end{figure}

Hence, we can consider the manifold of $2n$ complex dimensions, $S(-1)$. This will be our initial parameter space ($A$, in the terminology of Section \ref{sec:lifting-to-moduli-spaces}). Compare this to Section \ref{subsec:statement2}. There are two nonequivalent ways to apply LbP to this manifold, but both work the same way as the first application of LbP in Section \ref{subsec:statement2} (i.e. our $p^{-1}B$ term consists of a single marked point on $T$ converging to $\infty$). First, we move $y_1$ to $x_1=\infty$ in $T$ and obtain:

\begin{equation}
    \label{equation:1st-degreen}
    [\mu^{S(-1)}_{i+1}] = [\mu^{[\{x_1,y_1\}]}_{i+2}] + [\mu^{[\{x_1,\dots,x_n,y_1,\dots,y_n \}]}_{i+2}].
\end{equation}

Next, for each $j=1,\dots,n-1$ we note, by moving $y_{j+1} \rightarrow \infty$ in $T \subset S(2j-2)$, which yields a copy of $S(2j-1)$ (see Figure \ref{fig:thimble3}(iii),(iv),(v)) that:

\begin{equation}
    \label{equation:5th-degreen}
    [\mu^{[\{x_1,y_1\},y_2,x_2,\dots,y_j,x_j]}_{i+2j}] = [\mu^{[\{x_1,y_1\},y_2,x_2,\dots,y_j,x_j,y_{j+1}]}_{i+2j+1}] + [\mu^{[\{x_1,y_1\},y_2,x_2,\dots,y_j,x_j,\{x_{j+1},\dots,x_n,y_{j+1},\dots,y_n\}]}_{i+2j+1}].
\end{equation}

Likewise, for each $j=1,\dots,n-2$, if we move $x_{j+1} \rightarrow \infty$ in $T \subset S(2j-1)$, which yields a copy of $S(2j)$ (see Figure \ref{fig:thimble3}(iii),(v),(vi)) then:

\begin{equation}
    \label{equation:6th-degreen}
    [\mu^{[\{x_1,y_1\},y_2,x_2,\dots,y_j,x_j,y_{j+1}]}_{i+2j+1}] = [\mu^{[\{x_1,y_1\},y_2,x_2,\dots,y_j,x_j,y_{j+1},x_{j+1}]}_{i+2j+2}] + [\mu^{[\{x_1,y_1\},y_2,x_2,\dots,y_j,x_j,y_{j+1},\{x_{j+1},\dots,x_n,y_{j+2},\dots,y_n\}]}_{i+2j+2}].
\end{equation}

One iteratively applies Equation \eqref{equation:5th-degreen}, followed by \eqref{equation:6th-degreen}. For the final step, observe that (moving $x_n$ to $\infty$ in $T \subset S(2n-3)$ as in Section \ref{subsec:statement1}):
\begin{equation}
    \label{equation:7th-degreen}
    [\mu^{[\{x_1,y_1\},y_2,x_2,\dots,y_{n-1},x_{n-1},y_{n}]}_{i+2n-1}] = [\mu^{[\{x_1,y_1\},y_2,x_2,\dots,y_{n-1},x_{n-1},y_{n},x_n]}_{i+2n}].
\end{equation}

We can put together \eqref{equation:5th-degreen}, \eqref{equation:6th-degreen} and \eqref{equation:7th-degreen} (observing that \eqref{equation:7th-degreen} appears as the $j=n-1$ term in the last summation), to obtain that:

\begin{equation}
    \label{equation:8th-degreen}
    [\mu^{[\{x_1,y_1\}]}_{i+1}] =\sum_{j=1}^{n-1} \biggl( [\mu^{[\{x_1,y_1\},y_2,x_2,\dots,y_j,x_j,\{x_{j+1},\dots,x_n,y_{j+1},\dots,y_n\}]}_{i+2j} + [\mu^{[\{x_1,y_1\},y_2,x_2,\dots,y_j,x_j,y_{j+1},\{x_{j+1},\dots,x_n,y_{j+2},\dots,y_n\}]}_{i+2j+1}]\biggr).
\end{equation}

Next, returning to $\mu^{S(-1)}_{i+1}$, we note that if the $p^{-1}B$ term here was instead obtained by moving $x_2$ to $\infty$ in $T = S(-1)$ (which we had fixed to be $x_1$ by our choice of ordering) we obtain:

\begin{equation}
    \label{equation:2nd-degreen}
    [\mu^{S(-1)}_{i+1}] = [\mu^{[\{x_1,x_2\}]}_{i+2}] + [\mu^{[\{x_1,\dots,x_n,y_1,\dots,y_n \}]}_{i+2}].
\end{equation}

Note that for $j\ge 2$, by moving $x_{j+1}$ to $\infty$ in $T \subset S(j-1)$ (see Figure \ref{fig:thimble3}(i),(ii)) we obtain a copy of $S(j)$ and an inductive step:
\begin{equation}
    \label{equation:3rd-degreen}
    [\mu^{[\{x_1,x_2\},x_3,\dots,x_j]}_{i+j}] = [\mu^{[\{x_1,x_1\},x_3,\dots,x_{j+1}]}_{i+j+1}] + [\mu^{\{x_1,x_2 \},x_3,\dots,x_j,\{x_{j+1},\dots,x_n,y_1,\dots,y_n \}}_{i+j+1}].
\end{equation}

In particular, compiling \eqref{equation:2nd-degreen} and \eqref{equation:3rd-degreen} for $j=2,\dots,n$ we obtain that:
\begin{equation}
    \label{equation:4th-degreen}
    [\mu^{S(-1)}_{i+1}] = [\mu^{[\{x_1,\dots,x_n,y_1,\dots,y_n \}]}_{i+2}]+ [\mu^{[\{x_1,x_2\},x_3,\dots,x_n]}_{i+n}] + \sum_{j=2}^{n-1} [\mu^{[\{x_1,x_2 \},x_3,\dots,x_j,\{x_{j+1},\dots,x_n,y_1,\dots,y_n \}]}_{i+j+1}].
\end{equation}

Now, if we substitute \eqref{equation:4th-degreen} and \eqref{equation:8th-degreen} into \eqref{equation:1st-degreen}, then we obtain the final result of LbP:

\begin{equation}
    \label{equation:final-degreen}
    \begin{array}{l}
    [\mu^{[\{x_1,x_2\},x_3,\dots,x_n]}_{i+n}] + \sum_{j=2}^{n-1} [\mu^{[\{x_1,x_2 \},x_3,\dots,x_j,\{x_{j+1},\dots,x_n,y_1,\dots,y_n \}]}_{i+j+1}]= \\
    \sum_{j=1}^{n-1} \biggl( [\mu^{[\{x_1,y_1\},y_2,x_2,\dots,y_j,x_j,\{x_{j+1},\dots,x_n,y_{j+1},\dots,y_n\}]}_{i+2j} + [\mu^{[\{x_1,y_1\},y_2,x_2,\dots,y_j,x_j,y_{j+1},\{x_{j+1},\dots,x_n,y_{j+2},\dots,y_n\}]}_{i+2j+1}]\biggr).\end{array}
\end{equation}

Next, by iterating the lifting procedure, we need to interpret each of the $\Psi([\mu_i^f])$ for the pseudocycles (indeed, embedded submanifolds) above, and we know how to relate them via Section \ref{subsec:lifting-LbP}. We have already described (the various) $A$, and we pick the degree condition $\lambda=n$. The choice of Chern signature for the various iterative steps will be induced by the only possible choice of Chern signature on $S(-1)$ such that $\lambda = n$.

We note that in the moduli spaces that we consider for $\Psi([\mu^{[\{x_1,x_2\},x_3,\dots,x_n]}_{i+n-1}])$, as in the case of Section \ref{subsec:statement2} we can choose a homotopy of the data on the bubble tree such that the almost complex structure is as in Properties \ref{properties:J} and the Hamiltonian perturbation vanishes. Then the holomorphic maps on $S(n-1)_{0},\dots,S(n-1)_n$ must necessarily be constant. This is because the intersection of the thimble with $M_{\dagger}$ is $n$, and so in particular this must also be the intersection number with $M$. If any one of the holomorphic maps on $S(n-1)_{0},\dots,S(n-1)_n$ were nonconstant, this would imply that the limiting curve were to have degree $>n$. Hence, what we obtain is counts of holomorphic maps with thimbles with a degree $n$ tangency condition. This means $\Psi([\mu^{[\{x_1,x_2\},x_3,\dots,x_n]}_{i+n-1}])$ is the $u^{i+n-1}$ term of some Borman-Sheridan class $\tilde{s}^{(n)}_{eq}$, multiplied by $n!$ for the ambiguity in choice of $y_1,\dots,y_n$. For our purposes, we can take the definition $$n! \cdot \tilde{s}^{(n)}_{eq} := \sum_i \Psi([\mu^{[\{x_1,x_2\},x_3,\dots,x_n]}_{i}]) u^i.$$ 


Next, we consider $[\mu^{[\{x_1,x_2 \},x_3,\dots,x_j,\{x_{j+1},\dots,x_n,y_1,\dots,y_n \}]}_{i+j}]$. After a homotopy of the data as in the previous paragraph, if we consider the moduli spaces for the associated operation, then the holomorphic maps on $S(j)_0,\dots,S(j)_{j-1}$ must be constant, again for reasons of degree (in particular, the map on $S(j)_j$ has intersection number $n$ with $M_{\dagger}$). Thus we obtain some multiple of the equivariant PSS-map composed with a Gromov-Witten invariant $\xi^{j}$, which is a pseudocycle in $E$ associated to a GW invariant of total degree $n$, with a degree $j$ tangency to $M$ one point, an intersection with $M$ at a fixed marked point, and $n-j-1$ intersections with $M$ at freely moving marked points. The multiple, due to the ambiguity in choice of $y_1,\dots,y_n$, is $n!$. 

After a homotopy of the data again as above, and concurrently a homotopy of the intersection conditions of each of $y_2,\dots,y_j$ from $M_{\dagger}$ to respectively $M_{\dagger_2},\dots,M_{\dagger_j}$, one observes that the operation associated to $[\mu^{[\{x_1,y_1\},y_2,x_2,\dots,y_j,x_j,\{x_{j+1},\dots,x_n,y_{j+1},\dots,y_n\}]}_{i+2j}$ consists of (a multiple of):
\begin{enumerate}
    \item starting with $z^{(1)}$,
    \item  performing $(j-1)$ total iterations of: take the cup product with $[M_{\dagger_2}] = [M]$ (recalling that holomorphic maps must be constant on the spheres that do not contain $x_1,\dots,x_n$, as in the description of $\mu^{F''}_i$ in Section \ref{subsec:statement2} arising from the induced Chern signature) and then applying $*_F^{(1)}[M]$,
    \item  taking $*_F^{(n-j)} [M]$
    \item restricting to $E$ (as the pseudocycle we thus constructed cannot, for reasons of Chern number i.e. degree, land in $M$, so this forms a pseudocycle)
    \item   applying $PSS_{eq}$
\end{enumerate}
The multiple is $(n-j)! (n-j-1)!$, arising from the choice of ordering of $x_{j+2},\dots,x_n$, and the choice of each of $y_{j+1},\dots,y_{n}$.

Similarly to the previous paragraph, after a homotopy of the data, and concurrently of the intersection conditions of $y_2,\dots,y_{j+1}$ to become $M_{\dagger_2},\dots,M_{\dagger_{j+1}}$, the operation associated to $[\mu^{[\{x_1,y_1\},y_2,x_2,\dots,y_j,x_j,y_{j+1},\{x_{j+1},\dots,x_n,y_{j+2},\dots,y_n\}]}_{i+2j+1}]$ consists of (a multiple of):
\begin{enumerate}
    \item starting with $z^{(1)}$,
    \item  performing $(j-1)$ total iterations of: take the cup product with $[M]$ and then apply $*_F^{(1)}[M]$,
    \item  applying the cup product with $[M_{\dagger_{j+1}}] = [M]$, 
 \item taking $*_F^{(n-j)} [M]$, 
    \item restricting to $E$ 
    \item   applying $PSS_{eq}$
\end{enumerate}

 The multiplicative factor is $(n-j-1)! \cdot (n-j-1)!$, arising from choice of $x_{j+2},\dots,x_n$ and the choice of $n-j$ different intersections for each of $y_{j+2},\dots,y_n$.

In aggregate, repeated applications of Theorem \ref{theorem:lift-to-moduli} yield the following relation, where we note that for brevity $\clubsuit (x) := (x \cup [M]) *_F^{(1)} [M]$:

\begin{equation}
    \label{equation:final-degreen-operations}
    \begin{array}{l}
     n! \cdot \Biggl( u^{n} \tilde{s}^{(n)}_{eq} + PSS_{eq} \biggl( \sum_{j=2}^{n-1}  u^{2n-j-1} \xi^j \biggr) \Biggr) = \\
    PSS_{eq} \Biggl( \sum_{j=1}^{n-1} \biggl( (n-j)! (n-j-1)! u^{2n-2j} \cdot \Bigl( \clubsuit^{\circ j-1} (z^{(1)})\Bigr) *_F^{(n-j)}[M]\Biggr|_E \Biggr) \\ +PSS_{eq} \Biggl( \sum_{j=1}^{n-1} ((n-j-1)!)^2  u^{2n-2j-1} \cdot  \Bigl( \clubsuit^{\circ j-1}(z^{(1)}) \cup [M] \Bigr) *_F^{(n-j)}[M] \biggr)\Biggr|_E  \Biggr).\end{array}
\end{equation}

It should be briefly noted here that we can rearrange this to the following form:

\begin{theorem}
\label{theorem:general-solution}
\begin{equation}
\begin{array}{l} n! \cdot \Biggl( u^{n} \tilde{s}^{(n)}_{eq} + PSS_{eq} \biggl(  \sum_{j=2}^{n-1}  u^{2n-j-1} \xi^j \biggr) \Biggr) = \\ PSS_{eq} \Biggl( \sum_{j=1}^{n-1} ((n-j-1)!)^2 \cdot u^{2(n-j)-1} \biggl( \Bigl(\clubsuit^{\circ j-1} (z^{(1)})\Bigr) \cup \Bigl( (n-j) u + [M] \Bigr) \biggr) *_F^{(n-j)} [M]   \Biggr|_E   \Biggr). \end{array}
\end{equation}
\end{theorem}


\begin{remark}
Note that in the general case, unlike for $n=1,2$, we obtain terms of the form $u^{j}$ for $j > n$ on both sides of the equation. These should presumably cancel using some technology of Gromov-Witten invariants.
\end{remark}

\appendix


\section{Orientations for moduli spaces of degree 2 curves}
\label{sec:orientations}

We will confirm the signs for Equation \eqref{equation:paul-conjecture} in Section \ref{subsec:statement2}. Indeed, because \eqref{equation:4th-degree2} is as from Section \ref{subsec:statement1}, and the two relations \eqref{equation:1st-degree2} and \eqref{equation:2nd-degree2} are are related by a biholomorphism swapping $z_2$ and $z_3$, without loss of generality we only need to confirm the signs for the relation \eqref{equation:2nd-degree2}. We begin with the general version of LbP: \begin{equation}
    \label{equation:2nd-degree2-nosigns}   [\mu^A_{i-1} ] = \epsilon(F) [\mu^F_i] + n [\mu^{p^{-1} B_2}_i]
\end{equation}

In order to demonstrate that we obtain the correct signs, we first observe that there is a projection \begin{equation}\label{equation:projection}(T \setminus \{ \infty \})^3 \setminus \Delta_3 \rightarrow (T \setminus \{ \infty \} )^2 \setminus \Delta_2, \ (\infty,z_2,z_3,z_4) \mapsto (\infty,z_2,z_3),\end{equation} which extends to the compactification, perhaps collapsing some unstable components. Here, $\Delta_i \in T^i$ consists of all tuples where at least two points collide. We note that this projection is holomorphic, hence orientation preserving. We will therefore, for simplicity, work with the image of the projection, so $A'$ is the image of $A$ under \eqref{equation:projection}, $$A' = \text{Bl}^{\bC}_{\{ (\infty,\infty) \}} T^2, \quad W' := \text{Bl}^{\bR}_{F'}(A)$$ in which $F'$ is the exceptional divisor of $A'$ (the image of the projection of $F$ under \eqref{equation:projection}), and $E_{F'}$ the boundary in $W'$ associated to the unit normal bundle of $F'$. Then the open stratum of $A'$ is in Figure \ref{fig:thimble}(D) and $F'$ is in Figure \ref{fig:thimble}(E). 

Let $B' := \overline{\{ [x,\infty] = ([x],\infty) : [x] \in (T \setminus \{ \infty \})/S^1 \}} \subset A'/S^1$. In particular, $$p^{-1}B' :=\{ (x,\infty) : x \in T \setminus \{ \infty \} \},$$ which is the image of the projection of $p^{-1}B_2$ under \eqref{equation:projection}. See Figure \ref{fig:thimble2}(G). To see that the removal of $p^{-1}B'$ trivialises the $S^1$-bundle, we will write down a section of the bundle $$W' \setminus p^{-1}B' \rightarrow \widecheck{W'} \setminus B'.$$ Firstly, we notice that $E_{F'} \cong S^3$ and $E_{F'} \rightarrow F'$ is the Hopf fibration. Further, $B' \cap F'$ is a single point, and $E_{F'} \cap p^{-1}(B')$ is a copy of $S^1$. Hence, $E_{F'} \setminus (E_{F'} \cap p^{-1}B') = S^3 \setminus S^1 \cong D^2 \times S^1$. Further, in a collar neighbourhood of $E_{F'} \setminus (E_{F'} \cap p^{-1}B') \subset W'$, we observe that we have a manifold of the form $D^2 \times S^1 \times (-\epsilon,0]$. Away from $F'$, we define: $$s([x,y]) = \begin{cases}\begin{array}{l} (x |y| y^{-1},|y|) \text{ if } x \neq \infty, \\ (\infty,|y|) \text{ if } x = \infty \end{array}\end{cases}$$ which is well defined because $y \neq \infty$. 

We also observe that $s$ extends to a natural section $s_{B'}$ over $\text{Bl}_{B'} (\widecheck{W'})$, by adding in its values over $E_{B'}$ (the unit normal bundle of $B' \subset \widecheck{W}$) as follows. Firstly, we observe that locally $E_{B'}$ may be parametrised as $(x,Re^{i \theta})$ for $x \in B'$, for $\theta \in S^1$ and for some large $R$. Then $s_{B'}|_{E_{B'}} : E_{B'} \rightarrow p^{-1}B' $, where we define  \begin{equation} \label{eq:section-on-infty} s_{B'}[x,R e^{i \theta}] = (x e^{- i \theta}, \infty),\end{equation} which we observe is well-defined.

We obtain by LbP that $$[\mu_i^{A'} ] = \epsilon(F') [\mu_i^{F'}] + n' [\mu_i^{p^{-1}B'} ],$$ and then must determine the weights. First we show that $\epsilon(F') = 1$. Observe that $W'$ (with the natural complex orientation) is orientedly diffeomorphic to a closed ball in $\bC^2$ (with the natural complex orientation), and the boundary of this ball is outwardly oriented. Hence $\epsilon(F') = +1$.

Next, we calculate the weight $n'$ around $B'$. We will pick local coordinates. If we use $S^1$-equivariant coordinates (where on the left hand side the action is multiplication on the first coordinate) $$S^1 \times (0,\infty) \times S^1 \times (0,\infty) \rightarrow (T \setminus \{ \infty \}) \times (T \setminus \{ \infty \}), \quad (e^{2 \pi i \theta},r, e^{2 \pi i \phi},\rho) \mapsto (r e^{-2 \pi i \theta}, \rho, e^{-2 \pi i(\phi \theta)}),$$ then this induces the natural orientation on $T \times T$ (the differential of the transition function from the usual choice is a block-diagonal matrix, with each $2 \times 2$ matrix having positive determinant). This implies that, using our given choice of the orientation of $\widecheck{W'}$, we have oriented local coordinates \begin{equation} \label{eq:coordinates-check-W} (0,\infty) \times S^1 \times (0,\infty) \rightarrow \widecheck{W}, \quad (r,  e^{2 \pi i \phi},\rho) \mapsto [r,\rho e^{-2 \pi i \phi}].\end{equation} We can extend this to a boundary (i.e. $E_{B'}$) as $\rho \rightarrow \infty$. The local coordinates on $E_{B'}$ (which are chosen so that the coordinates \eqref{eq:coordinates-check-W} are the product of the local coordinates on $E_{B'}$ with an outward-pointing normal for the boundary at $\infty$, i.e. positively-oriented radial coordinates $\rho$) are \begin{equation} \label{eq:coordinates-E-B}(e^{2 \pi i \phi},r) \mapsto [r, (e^{2 \pi i \phi},\infty)] \in E_{B'}.\end{equation} We notice because we swapped the $e^{2 \pi i \phi}$ and $r$ coordinates, we precompose the $S^1$-coordinate with the complex conjugation to maintain orientation. Further, there is the natural complex coordinate on $p^{-1}B' \cong T \times \{ \infty \}$ which, on $T \setminus \{ \infty \} \cong (0,\infty) \times S^1$ becomes $(0,\infty) \times S^1 \rightarrow p^{-1} B', \quad (r, e^{i \psi}) \mapsto (r e^{i \psi}, \infty)$. Relabelling this in such a way as it matches the orientation of $B$ from \eqref{eq:coordinates-E-B} parametrised by $r \in (0,\infty)$, we obtain coordinates \begin{equation}\label{eq:coordinates-p-1-B} S^1 \times (0,\infty) \rightarrow p^{-1}B', \quad (e^{2 \pi i \psi}, r) \mapsto (re^{- 2 \pi i \psi}, \infty).\end{equation}

Putting together \eqref{eq:coordinates-E-B}, \eqref{eq:section-on-infty}, and \eqref{eq:coordinates-p-1-B}, we see that in local coordinates one obtains: $$\xymatrix{
(e^{2 \pi i \phi},r) \in S^1 \times (0,\infty)
\ar@{->}^-{\simeq}[r]
\ar@{->}_-{}[d]
&
[r,( e^{2 \pi i \phi},\infty)] \in E_{B'}
\ar@{->}^-{s_{B'}}[d]
\\
(e^{2 \pi i \phi}, r) \in S^1 \times (0,\infty)
&
(r e^{- 2 \pi i \phi},\infty) \in p^{-1}B'
\ar@{->}^-{\simeq}[l]
}$$

In particular, by restricting $s_{B'}$ to a fibre we obtain that $n'=1$. 

Now, we notice that $\epsilon(F) = \epsilon(F')$ and $n = n'$. This is because on the open stratum \eqref{equation:projection} is simply a projection map, which is immediately holomorphic, and thus is orientation preserving on the submanifolds and the induced pseudocycle bordism obtained from LbP. Hence, projection preserves $\epsilon(F)$ and $n$. This shows that we have obtained the correct signs in equations \eqref{equation:1st-degree2} and \eqref{equation:2nd-degree2}. The same arguments hold for all of the calculations in Section \ref{subsec:statementn}, as we are just performing iterated instances of \eqref{equation:paul-conjecture}.

\bibliographystyle{plain}
\bibliography{biblio}

\end{document}